\documentclass{amsart}
\usepackage{amsmath,amscd,amssymb,combelow,color,enumitem,graphicx,float,comment}

\newtheorem{theorem}{Theorem}[section]

\newtheorem{proposition}[theorem]{Proposition}
\newtheorem{lemma}[theorem]{Lemma}
\newtheorem{corollary}[theorem]{Corollary}
\newtheorem{conjecture}[theorem]{Conjecture}
\newtheorem{definition}[theorem]{Definition}
\newtheorem{claim}[theorem]{Claim}
\newtheorem{example}[theorem]{Example}
\newtheorem{remark}[theorem]{Remark}

\def\fb{{\mathfrak{b}}}
\def\fg{{\mathfrak{g}}}
\def\fh{{\mathfrak{h}}}

\def\fn{{\mathfrak{n}}}

\def\fgl{{\mathfrak{gl}}}

\def\hgl{{\widehat{\fgl}}}

\def\BC{{\mathbb{C}}}

\def\BN{{\mathbb{N}}}

\def\BZ{{\mathbb{Z}}}

\def\CA{{\mathcal{A}}}
\def\CB{{\mathcal{B}}}

\def\CO{{\mathcal{O}}}

\def\CS{{\mathcal{S}}}

\def\CV{{\mathcal{V}}}

\def\ph{\varphi}

\def\sym{\textrm{sym}}

\def\UU{U_q(L\fg)}
\def\UUp{U_q^+(L\fg)}
\def\UUpm{U^\pm_q(L\fg)}

\def\UUm{U_q^-(L\fg)}
\def\UUg{U_q^\geq(L\fg)}
\def\UUl{U_q^\leq(L\fg)}

\def\tUU{\widetilde{U}_q(L\fg)}
\def\tUUp{\widetilde{U}_q^+(L\fg)}
\def\tUUpm{\widetilde{U}^\pm_q(L\fg)}

\def\tUUm{\widetilde{U}_q^-(L\fg)}
\def\tUUg{\widetilde{U}_q^\geq(L\fg)}
\def\tUUl{\widetilde{U}_q^\leq(L\fg)}

\def\ba{{\mathbf{a}}}
\def\bb{{\mathbf{b}}}

\def\br{{\mathbf{r}}}

\def\bs{\boldsymbol{\varsigma}}

\def\cc{{\mathbb{C}^I}}

\def\nn{{\mathbb{N}^I}}
\def\zz{{\mathbb{Z}^I}}

\def\bom{{\boldsymbol{\la}}}

\def\bpsi{{\boldsymbol{\psi}}}

\def\bom{{\boldsymbol{\omega}}}

\def\bm{{\boldsymbol{m}}}
\def\bn{{\boldsymbol{n}}}

\def\bx{{\boldsymbol{x}}}

\def\b0{{\boldsymbol{0}}}
\def\bone{{\boldsymbol{1}}}

\def\loccit{\emph{loc.~cit.~}}
\def\loccitt{\emph{loc.~cit.}}

\def\wI{\widehat{I}}

\def\UUaff{U_q(\widehat{\fg})_{c=1}}

\def\UUaffg{U_q(\widehat{\fb}^+)_{c=1}}

\def\UUsh{U_q(L\fg)^{\mu}}

\def\wI{\widehat{I}}

\def\bx{\boldsymbol{x}}

\def\vac{|\varnothing\rangle}

\def\bone{{\boldsymbol{1}}}

\def\ord{\textbf{ord }}
\def\eord{\textbf{\emph{ord }}}

\def\op{\text{op}}

\def\sh{\text{sh}}
\def\esh{\emph{sh}}

\begin{document}

\title[Borel and shifted category $\CO$]{\Large{\textbf{Borel and shifted category $\CO$}}} 

\author[David Hernandez and Andrei Negu\cb t]{David Hernandez and Andrei Negu\cb t}

\address{Universit\'e Paris Cit\'e, Sorbonne Universit\'e, CNRS, IMJ-PRG, F-75013 Paris, France}

\email{david.hernandez@imj-prg.fr}

\address{École Polytechnique Fédérale de Lausanne (EPFL), Lausanne, Switzerland \newline \text{ } \ \ Simion Stoilow Institute of Mathematics (IMAR), Bucharest, Romania} 

\email{andrei.negut@gmail.com}
	
\begin{abstract} We prove a precise relation between simple modules in the Borel category $\CO$ and the shifted category $\CO$ for a symmetrizable Kac-Moody Lie algebra. \end{abstract}

\maketitle

\bigskip

\tableofcontents

\section{Introduction}
\label{sec:intro}

\medskip

\subsection{The Borel category $\CO$}
\label{sub:borel intro}

Consider a symmetrizable Kac-Moody Lie algebra $\fg$, with a set $I$ of simple roots. We write $(d_{ij})_{i,j \in I}$ for the symmetrized generalized Cartan matrix associated to $\fg$, see \eqref{eqn:cartan matrix}. Fix $q \in \BC^*$ not a root of unity and an enlargement $\fh \supseteq \cc$ of the root lattice as in Remark \ref{rem:enlarge cartan}; the latter is a technical requirement that ensures the non-degeneracy of the bilinear form on the Cartan subalgebra. The representation theory of the quantum loop algebra
\begin{equation}
\label{eqn:quantum loop intro}
\UU = \BC \Big \langle e_{i,d}, f_{i,d}, \ph_{i,d'}^+, \ph_{i,d'}^-, \kappa_{\ba} \Big \rangle_{i \in I, d \in \BZ, d' \geq 0, \ba \in \fh} \Big/ \Big( \text{relations in Def. \ref{def:pre quantum loop}, \ref{def:quantum loop algebra}} \Big)
\end{equation}
has long been studied from various points of view, see for instance \cite{CP, HJ} and the many references in \cite{H25} for $\fg$ of finite type, and \cite{HLMS, N Cat} for $\fg$ of general symmetrizable type. The Borel category $\CO$ (defined by Jimbo and the first-named author) contains simple modules 
\begin{equation}
\label{eqn:borel simple module intro}
\Big(\text{Borel subalgebra of }\UU \Big) \curvearrowright L(\bpsi)
\end{equation}
which are indexed by so-called highest $\ell$-weights \footnote{If we enlarge the Cartan subalgebra as in Remark \ref{rem:enlarge cartan}, then simple modules are indexed by $(\bpsi,\bom)$ as in Remark \ref{rem:enlarge simple} and not merely by $\bpsi$. However, we choose not to include $\bom$ in our notation for brevity.}
\begin{equation}
\label{eqn:ell weight intro}
\bpsi = \left(\psi_i(z) = \sum_{d=0}^{\infty} \frac {\psi_{i,d}}{z^d} \in \BC[[z^{-1}]]^\times \right)_{i \in I}
\end{equation}
which are rational, i.e. each $\psi_i(z)$ is the power series expansion \footnote{We note that our convention is to expand $\ell$-weights in $z^{-1}$ as opposed from the more usual $z$, in order to ensure uniformity between our $z$ and the variables of shuffle algebras, see \eqref{eqn:big shuffle}.} of a rational function  in $z$ regular at $\infty$.

\medskip 

Let us write $\br = \ord \bpsi \in \zz$ for the $I$-tuple that keeps track of the orders of the poles of the rational functions $\psi_i(z)$ at $z=0$, and call it the order of $\bpsi$. Then there is (\cite{N Cat, N Char}) an isomorphism of vector spaces
\begin{equation}
\label{eqn:l decomposition intro}
L(\bpsi) \cong L^{\br} \otimes L^{\neq 0} (\bpsi)
\end{equation}
which underlies the decomposition of the $q$-character (long-known in special cases when  $\fg$ is of finite type \cite{FR, FH, HL Borel, FJMM, H Shifted, FH2}) as
\begin{equation}
\label{eqn:chi decomposition intro}
\chi_q(L(\bpsi)) = \chi^{\br} \cdot \chi_q(L^{\neq 0}(\bpsi))
\end{equation}
The factor $\chi^{\br}$ is an ordinary character (as opposed from a $q$-character) that only depends on $\br = \ord \bpsi$, and it has been computed in \cite{N Char}, in accordance with conjectures of \cite{MY, W1}. In particular, we recover two limit cases that were already known for finite type $\fg$:

\medskip

\begin{itemize} 

\item If $L(\bpsi)$ is a finite-dimensional representation of the entire quantum loop algebra, then 
there is only one factor $\chi_q(L(\bpsi)) = \chi_q(L^{\neq 0}(\bpsi))$ by \cite{FR}. 

\medskip

\item If $\bpsi$ is a polynomial in $z^{-1}$, then $\chi_q(L^{\neq 0}(\bpsi)) = [\bpsi]$ has only one term by \cite{FH}, and so $\chi_q(L(\bpsi)) = \chi^{\br} [\bpsi]$.

\end{itemize} 

 Our main interest is to calculate the second factor in \eqref{eqn:chi decomposition intro} by gaining an understanding of the vector space $L^{\neq 0}(\bpsi)$ itself,  for any highest $\ell$-weight $\bpsi$. We treat general symmetrizable Kac-Moody Lie algebras $\fg$, but our results are also new for $\fg$ of finite type, as we handle all simple modules in the category $\mathcal{O}$.

\medskip

\subsection{The shifted category $\CO^{\sh}$}
\label{sub:shifted intro}

By analogy with the Borel category $\CO$, the first-named author introduced the shifted category $\CO^{\sh}$ consisting of modules of the shifted quantum loop algebra
\begin{equation}
\label{eqn:shifted quantum loop intro}
\UUsh = \BC \Big \langle e_{i,d}, f_{i,d}, \ph_{i,d'}^+,  \ph_{i,d'}^-, \kappa_{\ba} \Big \rangle_{i \in I, d \in \BZ, d' \geq 0,\ba \in \fh} \Big/ \Big( \text{relations in Definition \ref{def:shifted}} \Big)
\end{equation}
While the Borel category $\mathcal{O}$ is well designed to study quantum integrable models thanks to the transfer-matrix construction \cite{FH, FH2}, the shifted category $\mathcal{O}^{\sh}$ fits very well in the framework of cluster categorification \cite{GHL}. In contrast, it is not known how to assign transfer-matrices to any module in $\mathcal{O}^{\sh}$, and there are difficulties to obtain direct cluster categorification from the whole category $\mathcal{O}$ (as tensor products of simple modules in $\mathcal{O}$ are not always of finite length). In this picture, it is thus important to understand the precise relation between the two categories of modules.

\medskip 

The algebra \eqref{eqn:shifted quantum loop intro} was defined in \cite{FT} for any integral coweight
\begin{equation}
\label{eqn:r mu intro}\
\mu = \sum_{i \in I} r_i \omega_i^\vee
\end{equation}
where $\br = (r_i)_{i \in I} \in \zz$. Moreover, for $\fg$ of finite type and for any rational $\ell$-weight $\bpsi$ with $\br = \ord \bpsi$, a simple module
\begin{equation}
\label{eqn:shifted simple module intro}
\UUsh \curvearrowright L^{\sh}(\bpsi)
\end{equation}
was constructed in \cite{H Shifted} (by \loccitt, $\UUsh$ has a non-trivial finite-dimensional module if and only if $\mu$ is codominant). In the present paper, such simple modules will be defined for an arbitrary symmetrizable Kac-Moody Lie algebra $\fg$.

\medskip 

When $\fg$ is of finite type and $L^{\sh}(\bpsi)$ is finite-dimensional, it was quickly recognized in \cite{H Shifted} that $\chi_q(L^{\sh}(\bpsi))$ matches the second factor in the right-hand side of \eqref{eqn:chi decomposition intro}. In the present paper, we establish this fact for a general $L(\bpsi)$ by lifting it from an equality of numbers to an isomorphism of vector spaces.

\medskip

\begin{theorem}
\label{thm:main intro}

For any symmetrizable Kac-Moody Lie algebra $\fg$ and any rational $\ell$-weight $\bpsi$, we have a vector space isomorphism 
\begin{equation}
\label{eqn:main intro}
L^{\neq 0}(\bpsi) \cong L^{\esh}(\bpsi)
\end{equation}
which preserves the natural gradings by $\bn \in \nn$ and $\bx \in (\BC^*)^{\bn}$ on both sides (see \eqref{eqn:decomposition x} and \eqref{eqn:decomposition x shifted}). Therefore, \eqref{eqn:main intro} descends to an equality of $q$-characters 
\begin{equation}
\label{eqn:main intro q-characters}
\chi_q(L^{\neq 0}(\bpsi)) = \chi_q(L^{\esh}(\bpsi)).
\end{equation}

\end{theorem}

\medskip

The result above not only holds for any symmetrizable Kac-Moody Lie algebra $\fg$, but it is also new for $\fg$ of finite type, as it is established for any simple module $L^{\sh}(\bpsi)$ (in particular $\mu$ is not necessarily codominant). Our methods are also different from the ones in the literature, see for instance \cite{HJ, H Shifted}.

\medskip

The isomorphism of vector spaces \eqref{eqn:main intro} lifts to one of modules for
$$
\UUm = \BC \Big \langle  f_{i,d} \Big \rangle_{i \in I, d \in \BZ} \Big/ \Big( \text{relations} \Big)
$$
which is a common subalgebras to both \eqref{eqn:quantum loop intro} and \eqref{eqn:shifted quantum loop intro}. However, since the the gluing between the positive and negative halves differs between the algebras \eqref{eqn:quantum loop intro} and \eqref{eqn:shifted quantum loop intro}, we cannot upgrade the isomorphism \eqref{eqn:main intro} any further.

\medskip

\subsection{$QQ$-systems}

Another important application of our results is given by the simple modules corresponding to $\tilde{Q}$-variables in $QQ$-systems \cite{FH2}. 

\medskip 

Recall that the ODE/IM correspondence gives a surprising relation between functions associated to Schr\"odinger differential operators and the spectrum of quantum systems called ``quantum KdV". Feigin-Frenkel \cite{FF} have proposed a large generalization of this correspondence in terms of Langlands duality. This open conjecture is a fruitful source of inspiration.
In particular, a remarkable system of relations (the $QQ$-system) was observed to be satisfied
by spectral determinants of certain solutions of affine opers \cite{mrv}. Then, motivated by the general Feigin-Frenkel conjecture, it was proved in 
\cite{FH2} that this $QQ$-system has a solution in the Grothendieck of the Borel category $\mathcal{O}$ (when the underlying Lie algebra is of finite type). The solution is described in terms of simple classes up to multiplicative constants
(the renormalization factors, which are to be computed). 

\medskip

Our results give the precise renormalization factors to write the $QQ$-system in the Grothendieck ring of the Borel category $\mathcal{O}$.
Indeed, a solution of the $QQ$-system exists in the category $\mathcal{O}^{\sh}$ {\it without any renormalization factor} by \cite{H Shifted}.
The relation between representations of shifted quantum loop algebras and Borel algebras that we establish here is the missing piece to compute the renormalization factors.

\medskip 

Note that the $QQ$-systems are closely related to the Bethe Ansatz relations in quantum integrable systems \cite{mrv} and to exchange relations in certain monoidal categorifications of cluster algebras \cite{GHL}. Hence, the precise $QQ$-system in the Grothendieck group of category $\CO$ established in the present paper opens the way to new developments in these directions as well.

\medskip

More generally, we establish a ring isomorphism between the Grothendieck rings of $\mathcal{O}$ and $\mathcal{O}^{\sh}$. This allows to formulate the conjectures in \cite{FH3} on generalized $QQ$-systems in terms of the Borel category $\mathcal{O}$. 

\medskip

With this in mind, one of the main applications of our results is Theorem \ref{thm:qq}: there is an explicit solution of the $QQ$-system in the Borel category $\mathcal{O}$. This result generalizes that of \cite{FH2} from finite type to an arbitrary symmetrizable Kac-Moody Lie algebra $\fg$.

\medskip

\subsection{Shuffle algebras}
\label{sub:shuffle intro}

There are many tasks that go into establishing Theorem \ref{thm:main intro}: defining Borel subalgebras of $\UU$ for arbitrary symmetrizable Kac-Moody Lie algebras $\fg$, explicitly constructing the vector spaces $L(\bpsi)$, $L^{\neq 0}(\bpsi)$ and $L^{\sh}(\bpsi)$ and establishing the coincidence of the latter two. All these tasks can be performed using the techniques of shuffle algebras (\cite{E, FO}). In a nutshell, there is a subspace
$$
\CS ^- \subseteq \CV = \bigoplus_{\bn \in \nn} \frac {\BC [z_{i1}^{\pm 1},\dots,z_{in_i}^{\pm 1} ]^{\sym}_{i \in I}}{\prod^{\text{unordered}}_{i \neq j, a, b} (z_{ia} - z_{jb})}
$$
such that $\CS^- \cong \UUm$, and a subalgebra 
$$
\CS^-_{<\b0} \subset \CS^-
$$
such that $\CS^-_{<\b0} \cong \UUm \cap \UUaffg$ in finite types. This allowed the second-named author to prove the following isomorphisms in \cite{N Cat, N Char}
\begin{align}
L(\bpsi) &= \CS^-_{<\b0} \Big / J(\bpsi) \label{eqn:l intro} \\
L^{\neq 0}(\bpsi) &= \CS^-_{<\b0} \Big / J^{\neq 0}(\bpsi) \label{eqn:l neq intro} 
\end{align}
where $J(\bpsi)$ and $J^{\neq 0}(\bpsi)$ are certain subsets of rational functions that we recall in Subsection \ref{sub:simple}. We also have
\begin{equation}
\label{eqn:l bar intro}
L^{\neq 0}(\bpsi) \cong \CS^- \Big / \bar{J}^{\neq 0}(\bpsi)
\end{equation}
where $\bar{J}^{\neq 0}(\bpsi)$ is defined in Lemma \ref{lem:iso}. Our main technical result is the following

\medskip

\begin{theorem}
\label{thm:technical intro}

For any symmetrizable Kac-Moody Lie algebra $\fg$ and any rational $\ell$-weight $\bpsi$, let $\mu$ and $\br = \eord \bpsi$ be related by \eqref{eqn:r mu intro}. Then
\begin{equation}
\label{eqn:technical intro}
L^{\esh}(\bpsi) = \CS^- \Big / J^{\esh}(\bpsi) 
\end{equation}
(where $J^{\esh}(\bpsi)$ is given in Definition \ref{def:ideal}) is the unique up to isomorphism simple $\UUsh$-module generated by a single vector $\vac$ satisfying relations \eqref{eqn:simple module shifted property}-\eqref{eqn:simple module shifted property 2}.

\end{theorem}

\medskip

 The following fact is non-trivial
$$
\bar{J}^{\neq 0}(\bpsi) =  J^{\sh}(\bpsi) 
$$
and will be proved in \eqref{eqn:easy}. Comparing \eqref{eqn:l bar intro} with \eqref{eqn:technical intro} implies Theorem \ref{thm:main intro}. 

\medskip

\subsection{Acknowledgements}
The first-named author gratefully acknowledges the support of the Agence Nationale de la Recherche (grant ANR-24-CE40-3389). The second-named author gratefully acknowledges the support of the Swiss National Science Foundation grant 10005316.

\bigskip

\section{The Borel category $\CO$}
\label{sec:borel}

\medskip

 We recall $\UU$ and its Borel subalgebra for a complex semisimple Lie algebra $\fg$, and then generalize these notions to arbitrary symmetrizable Kac-Moody Lie algebras. We then review the Borel category $\CO$ defined in \cite{HJ, N Cat}, as well as the explicit construction of simple modules in this category from \loccit

\medskip 

\subsection{Basic notations}
\label{sub:basic notations}

The set $\BN$ will contain 0 throughout this paper. Fix $q \in \BC^*$, not a root of unity. We fix $h\in\BC$ satisfying $q = e^h$ so that the complex powers of $q$ are well-defined. To a finite set $I$ and a Cartan matrix
\begin{equation}
\label{eqn:cartan matrix}
C = \left(c_{ij} = \frac {2d_{ij}}{d_{ii}} \in \BZ\right)_{i,j \in I}
\end{equation}
one can associate a complex semisimple Lie algebra $\fg$. The numbers $d_{ij} = d_{ji}$ for $i\neq j$ must be non-positive integers, while the numbers $d_{ii}$ must be even positive integers. We will write
$$
d_i = \frac {d_{ii}}2
$$
for all $i \in I$. The root lattice of the Lie algebra $\fg$ will be identified with $\zz$ in what follows, and we will write
\begin{equation}
\label{eqn:bilinear form}
\cc \otimes \cc \xrightarrow {(\cdot,\cdot)} \BC, \qquad (\bs^i, \bs^j) = d_{ij}
\end{equation}
for the complexified symmetric invariant bilinear form, where 
$$
\bs^i = \underbrace{(0,\dots,0,1,0,\dots,0)}_{1 \text{ on }i\text{-th position}} \in \nn
$$
represents the $i$-th simple root. We will consider the standard partial order
$$
\bm \leq \bn \quad \Leftrightarrow \quad \bn - \bm \in \nn
$$
and write $\bm < \bn$ if $\bm \leq \bn$ and $\bm \neq \bn$. Let $\b0 = (0,\dots,0)$, $\bone=(1,\dots,1)$ and
$$
|\bm| = \sum_{i \in I} m_i
$$
for all $\bm = (m_i)_{i \in I}$.

\medskip

\begin{remark}
\label{rem:enlarge cartan}

Later in our paper, we will generalize complex semisimple Lie algebras to symmetrizable Kac-Moody Lie algebras. In this level of generality, there exist situations (such as type $\widehat{A}_2$) when the Cartan matrix \eqref{eqn:cartan matrix} is singular and hence the bilinear form \eqref{eqn:bilinear form} is degenerate. In this case, we extend the bilinear form following \cite{Kac}: we choose a complex vector space $\fh$ and a non-degenerate pairing
\begin{equation}
\label{eqn:bilinear form extended}
\fh \otimes \fh \xrightarrow {(\cdot,\cdot)} \BC
\end{equation}
which extends \eqref{eqn:bilinear form} with respect to a henceforth fixed inclusion $\cc \hookrightarrow \fh$ (we abuse notation and write $\bs^i$ for the image of the $i$-th standard basis vector in $\fh$). It is well-known that the smallest such enlargement $\fh$ has dimension $2|I| - \emph{rank}(C)$.
    
\end{remark}

\medskip

\subsection{The quantum affine algebra}
\label{sub:quantum affine}

We will be interested in the quantum affine algebra of $\fg$, with trivial central element. If we let $\wI = I \sqcup 0$, this algebra is defined as
\begin{equation}
\label{eqn:quantum affine group}
\UUaff = \BC \Big \langle e_i, f_i, \kappa_i^{\pm 1} \Big \rangle_{i \in \wI} \Big/ \Big( [\kappa_i, \kappa_j] = 0 \text{ and other relations} \Big)
\end{equation}
The ``other relations" one imposes in the right-hand side will not be used in the present paper, and the interested reader can find them in \cite[Subsection 2.1]{HJ}. The algebra \eqref{eqn:quantum affine group} has a Borel subalgebra 
$$
\UUaffg \subset \UUaff
$$
generated by $\{e_i, \kappa_i^{\pm 1}\}_{i \in \wI}$.

\medskip

\begin{definition} 
\label{def:category o affine}

(\cite{HJ}) Consider the category $\CO$ of complex representations 
\begin{equation}
\label{eqn:rep affine}
\UUaffg \curvearrowright V
\end{equation}
which admit a decomposition
\begin{equation}
\label{eqn:weight decomposition}
V = \bigoplus_{\bom \in \cup_{s=1}^t (\bom^s - \nn) } V_{\bom}
\end{equation}
for finitely many $\bom^1,\dots,\bom^t \in \cc$, such that every weight space
\begin{equation}
\label{eqn:weight}
V_{\bom} = \left\{v \in V \Big| \kappa_i \cdot v = q^{(\bom, \bs^i)}v, \ \forall i \in I \right\}
\end{equation}
is finite-dimensional. If $v \in V_\bom$ as above, then we call $\bom$ the weight of $v$.

\end{definition}

\medskip

\subsection{The quantum loop algebra}
\label{sub:quantum loop}

In order to index simple modules in the category $\CO$ of Definition \ref{def:category o affine}, we follow in the footsteps of \cite{CP} and consider the non-trivial isomorphism (constructed by \cite{Dr} and proved by \cite{B, Da})
\begin{equation}
\label{eqn:isomorphism}
\Phi : \UUaff \xrightarrow{\sim} \UU
\end{equation}
where the object in the right-hand side is the quantum loop algebra
\begin{equation}
\label{eqn:quantum loop group}
\UU = \Big \langle e_{i,d} , f_{i,d}, \ph_{i,d'}^+, \ph_{i,d'}^- \Big \rangle_{i \in I, d \in \BZ, d' \geq 0} \Big/ \Big( [\ph_{i,s}^\pm, \ph_{j,t}^{\pm'}] = 0 \text{ and other relations} \Big).
\end{equation}
We will recall the full set of relations in $\UU$ in Definitions \ref{def:pre quantum loop} and \ref{def:quantum loop algebra}. The only properties of the isomorphism \eqref{eqn:isomorphism} that will be important to us are
$$
\Phi(\kappa_i) = \ph_{i,0}^+
$$
$$
\Phi \left(\UUaffg\right) \supset \{\ph_{i,0}^+, \ph_{i,1}^+,\dots\}_{i \in I} \cup  \{e_{i,0}, e_{i,1},\dots\}_{i \in I}.
$$
As such, we can consider modules in category $\CO$ and ask how their weight spaces further decompose into generalized eigenspaces for the commutative family of endomorphisms $\{\ph_{i,0}^+,\ph_{i,1}^+,\dots\}_{i \in I}$. This is best systematized by the following notion.

\medskip

\begin{definition} 
\label{def:l weight}

An \textbf{$\ell$-weight} is an $I$-tuple of invertible power series
\begin{equation}
\label{eqn:l weight}
\bpsi = \left(\psi_i(z) = \sum_{d = 0}^{\infty} \frac {\psi_{i,d}}{z^d} \in \BC[[z^{-1}]]^\times \right)_{i \in I} 
\end{equation}
If every $\psi_i(z)$ is the expansion of a rational function, then $\bpsi$ is called \textbf{rational}.

\end{definition}

\medskip

 Thus, given an $\ell$-weight $\bpsi$, we can do two things:

\medskip

\begin{itemize}[leftmargin=*]

\item for any module $\UUaffg \curvearrowright V$, consider the generalized eigenspace
\begin{equation}
\label{eqn:generalized eigenspace}
V_{\bpsi} = \left\{v \in V \Big| \left(\ph_{i,d}^+ - \psi_{i,d} \cdot \text{Id}_V \right)^N (v) = 0\text{ for any }i,d \text{ and for }N \gg 0 \right\}
\end{equation}

\item consider the simple module 
\begin{equation}
\label{eqn:simple module}
\UUaffg \curvearrowright L(\bpsi)
\end{equation}
generated by a single vector $\vac$ subject to the relations
\begin{equation}
\label{eqn:simple module relations}
e_{i,d}\cdot \vac = 0 \qquad \text{and} \qquad \ph_{i,d}^+ \cdot \vac = \psi_{i,d} \vac
\end{equation}
for all $i \in I, d \geq 0$. A simple module \eqref{eqn:simple module} was shown to exist and be unique up to isomorphism in \cite{HJ}, building upon the work of \cite{CP} on finite-dimensional modules. 
\end{itemize}

Moreover, we have the following.

\begin{theorem}\cite{HJ} The representation $L(\bpsi)$ lies in category $\CO$ if and only if $\bpsi$ is rational.
\end{theorem}

We henceforth work only with rational $\ell$-weights $\bpsi$.

\medskip

 The characters (i.e. the generating series of dimensions of weight spaces \eqref{eqn:weight}) of simple modules are not strong enough to distinguish between different $L(\bpsi)$'s. However, the $q$-characters of Frenkel-Reshetikhin (\cite{FR, HJ}, i.e. the generating series of dimensions of $\ell$-weight spaces \eqref{eqn:generalized eigenspace}) of simple modules do distinguish between them. We will denote these $q$-characters as follows, for any $V$ in category $\CO$
\begin{equation}
\label{eqn:q character}
\chi_q(V) = \sum_{\ell \text{-weights } \bpsi} \dim_{\BC} \left(V_{\bpsi} \right) [\bpsi]
\end{equation}
where $[\bpsi]$ are formal symbols that multiply component-wise (as $I$-tuples).

\medskip

\subsection{Kac-Moody Lie algebras}
\label{sub:kac-moody}

We will henceforth assume that $\fg$ is a symmetrizable Kac-Moody Lie algebra, or equivalently, drop the positive-definiteness requirement on the Cartan matrix \eqref{eqn:cartan matrix}. In this level of generality, we do not have a notion of 
$$
\UUaff \text{ and its Borel subalgebra } \UUaffg.
$$
However, there exists a notion of $\UU$, to which we will associate a replacement of the Borel subalgebra using shuffle algebra tools. We start with a two-step definition of the quantum loop algebra associated to $\fg$. In what follows, we will write
$$
e_i(z) = \sum_{d = -\infty}^{\infty} \frac {e_{i,d}}{z^d}, \qquad f_i(z) = \sum_{d = -\infty}^{\infty} \frac {f_{i,d}}{z^d}, \qquad \ph^\pm_i(z) = \sum_{d = 0}^{\infty} \frac {\ph^\pm_{i,d}}{z^{\pm d}}.
$$ 
and $q_i = q^{d_i}$. We assume that an enlargement $\fh \supseteq \cc$ as in Remark \ref{rem:enlarge cartan} is given, and define the following notion (which differs from that of \cite{N Cat} by the fact that it has a bigger Cartan subalgebra; symbols $\{\kappa_{\ba}\}_{\ba \in \fh}$ will always be considered to be $\BC$-additive, i.e. satisfy the formula 
\begin{equation}
\label{eqn:additive}
\kappa_{\alpha \ba+ \beta \bb} = \kappa_{\ba}^{\alpha} \kappa_{\bb}^{\beta}
\end{equation}
for all $\alpha,\beta \in \BC$ and $\ba,\bb \in \fh$; in all representations we will consider, all the $\kappa_{\ba}$'s will act diagonally with non-zero complex eigenvalues). For all $i,j \in I$, let
\begin{equation}
\label{eqn:def zeta}
\zeta_{ij}(x) = \frac {x - q^{-d_{ij}}}{x-1}.
\end{equation}

\medskip

\begin{definition}
\label{def:pre quantum loop}

The pre-quantum loop algebra associated to $\fg$ is
$$
\tUU = \BC \Big \langle e_{i,d}, f_{i,d}, \ph_{i,d'}^+ , \ph_{i,d'}^-, \kappa_{\ba} \Big \rangle_{i \in I, d \in \BZ, d' \geq 0, \ba \in \fh} \Big/
  \text{relations \eqref{eqn:rel 0 loop}-\eqref{eqn:rel 3 loop}}
$$
where we impose the following relations for all $i,j \in I$, $\ba \in \fh$, $\pm, \pm' \in \{+,-\}$:
\begin{equation}
\label{eqn:rel 0 loop}
  e_i(x) e_j(y) \zeta_{ji} \left( \frac yx \right) =\, e_j(y) e_i(x) \zeta_{ij} \left(\frac xy \right)
\end{equation}
\begin{equation}
\label{eqn:rel 1 loop}
  \ph_j^\pm(y) e_i(x) \zeta_{ij} \left(\frac xy \right) = e_i(x) \ph_j^\pm(y) \zeta_{ji} \left( \frac yx \right)
\end{equation}
\begin{equation}
\label{eqn:rel 1 loop bonus}
\kappa_{\ba} e_i(x) = e_i(x) \kappa_{\ba} q^{(\ba,\bs^i)}
\end{equation}
\begin{equation}
\label{eqn:rel 2 loop}
  \ph_{i}^{\pm}(x) \ph_{j}^{\pm'}(y) = \ph_{j}^{\pm'}(y) \ph_{i}^{\pm}(x), \quad \ph_{i,0}^{\pm} = \kappa_{\pm \bs^i}
\end{equation}
as well as the opposite relations \footnote{In other words, we replace any product $\dots e\ph e' \ph' \dots$ by $ \dots \ph' f' \ph f \dots$.} with $e$'s replaced by $f$'s, and finally the relation
\begin{equation}
\label{eqn:rel 3 loop}
  \left[ e_i(x), f_j(y) \right] =
  \frac {\delta_{ij}\delta \left(\frac xy \right)}{q_i - q_i^{-1}}  \Big( \ph_i^+(x) - \ph_i^-(y) \Big).
\end{equation}
In formula \eqref{eqn:rel 0 loop}, we clear denominators and obtain relations by equating the coefficients of all $x^{-d} y^{-d'}$ in the left and right-hand sides, while in \eqref{eqn:rel 1 loop} we expand in non-positive powers of $y^{\pm 1}$ and then equate coefficients. 

\end{definition}

\medskip

Note that we do not assume any Drinfeld-Serre type relations to hold in the pre-quantum loop algebra, and instead include (generalized versions of) such relations in Definition \ref{def:quantum loop algebra} below. More specifically, the quantum loop algebra $\UU$ will be defined as a particular quotient of the pre-quantum loop algebra $\tUU$, which we will introduce using the language of shuffle algebras. 

\medskip

We have the following shift automorphisms for all $\br = (r_i)_{i\in I} \in \zz$
\begin{equation}
\label{eqn:shift pre quantum}
\sigma_{\br} : \tUU \rightarrow \tUU, \quad e_{i,d} \mapsto e_{i,d+r_i}, f_{i,d} \mapsto f_{i,d-r_i}, \ph_{i,d}^\pm \mapsto \ph_{i,d}^\pm, \kappa_{\ba} \mapsto \kappa_{\ba}
\end{equation}

\medskip

\begin{remark}

We note that different authors use different normalizations for the Cartan elements, and we compare them here. In the present paper, we write
\begin{equation}
\label{eqn:phi to p}
\ph_j^\pm(y) = \kappa_{\bs^j}^{\pm 1} \exp \left(\sum_{u=1}^{\infty} \frac {p_{j,\pm u}}{u y^{\pm u}} \right)
\end{equation}
in terms of which \eqref{eqn:rel 1 loop} implies
\begin{equation}
\label{eqn:rel 2 loop k}
[p_{j,u}, e_i(x)] = e_i(x) x^u \left(q^{ud_{ij}} - q^{-ud_{ij}} \right)
\end{equation}
Therefore, the comparison between the generators $p_{j,u}$ and the $\tilde{h}_{i,u}$ of \cite[Section 9.2]{H Shifted} is given by the following formula for all $j\in I$ and $u \in \BZ \backslash 0$
$$
p_{j,u} = u(q-q^{-1}) \sum_{i \in I}  \tilde{h}_{i,u} C_{ij}(q^u)
$$
where 
\begin{equation}
\label{eqn:quantum Cartan matrix}
C_{ij}(x) = \frac {x^{d_{ij}}- x^{-d_{ij}}}{x^{d_i}-x^{-d_i}}
\end{equation}
is the modified quantum Cartan matrix (it coincides with the usual quantum Cartan matrix, which has entries $x^{d_i}+x^{-d_i}$ on the diagonal and $[c_{ij}]_x$ off the diagonal, for finite type $\fg$).
    
\end{remark}

\medskip

\subsection{The big shuffle algebra}
\label{sub:big shuffle}

We now review the trigonometric version (\cite{E}) of the Feigin-Odesskii shuffle algebra (\cite{FO}) associated to the Kac-Moody Lie algebra $\fg$. Consider the vector space of rational functions in arbitrarily many variables
\begin{equation}
\label{eqn:big shuffle}
\CV = \bigoplus_{\bn \in \nn} \CV_{\bn}, \quad \text{where} \quad \CV_{(n_i \geq 0)_{i \in I}} = \frac {\BC[z_{i1}^{\pm 1},\dots,z_{in_i}^{\pm 1}]^{\text{sym}}_{i \in I}}{\prod^{\text{unordered}}_{i \neq j} \prod_{a=1}^{n_i} \prod_{b=1}^{n_j} (z_{ia} - z_{jb})} 
\end{equation}
Above, ``sym" refers to color-symmetric rational functions, meaning that they are symmetric in the variables $z_{i1},\dots,z_{in_i}$ for each $i \in I$ separately (the terminology is inspired by the fact that $i \in I$ is called the color of the variable $z_{ia}$). We make the vector space $\CV$ into a $\BC$-algebra via the following shuffle product:
\begin{equation}
	\label{eqn:mult}
E( z_{i1}, \dots, z_{i n_i})_{i \in I} * E'(z_{i1}, \dots,z_{i n'_i})_{i \in I} = \frac 1{\bn!  \bn'!}\,\cdot
\end{equation}
$$
\textrm{Sym} \left[ E(z_{i1}, \dots, z_{in_i}) E'(z_{i,n_i+1}, \dots, z_{i,n_i+n'_i})
\prod_{i,j \in I} \mathop{\prod_{1 \leq a \leq n_i}}_{n_j < b \leq n_j+n_j'} \zeta_{ij} \left( \frac {z_{ia}}{z_{jb}}\right) \right]
$$
The word ``Sym" in \eqref{eqn:mult} denotes symmetrization with respect to the
\begin{equation*}
(\bn+\bn')! := \prod_{i\in I} (n_i+n'_i)!
\end{equation*}
permutations of the variables $\{z_{i1}, \dots, z_{i,n_i+n'_i}\}$ for each $i$ independently. The reason for formula \eqref{eqn:mult} is to ensure that there exist algebra homomorphisms
\begin{align}
&\widetilde{\Upsilon}^+ : \tUUp \rightarrow \CV, \qquad \quad e_{i,d} \mapsto z_{i1}^d \in \CV_{\bs^i} \label{eqn:upsilon plus} \\
&\widetilde{\Upsilon}^- : \tUUm \rightarrow \CV^{\text{op}}, \qquad f_{i,d} \mapsto z_{i1}^d \in \CV_{\bs^i}^{\text{op}} \label{eqn:upsilon minus}
\end{align}
where 
\begin{align*}
&\tUUp = \BC \Big \langle e_{i,d} \Big \rangle_{i \in I, d \in \BZ} \Big / ( \text{relation \eqref{eqn:rel 0 loop}}) \\
&\tUUm = \BC \Big \langle f_{i,d} \Big \rangle_{i \in I, d \in \BZ} \Big / ( \text{opposite of relation \eqref{eqn:rel 0 loop}}) 
\end{align*}
By standard results on triangular decompositions, we have
$$
\tUU = \tUUp \otimes  \BC \left [\ph_{i,1}^\pm, \ph_{i,2}^\pm, \dots, \kappa_{\ba} \right ]_{i \in I, \ba \in \fh} \otimes \tUUm
$$
with the natural identification of generators. Note that $\widetilde{\Upsilon}^\pm$ intertwine the shift automorphisms \eqref{eqn:shift pre quantum} with
\begin{align} 
&\sigma_{\br} : \CV \rightarrow \CV, \qquad \qquad R(z_{i1},\dots,z_{in_i}) \mapsto R(z_{i1},\dots,z_{in_i}) \prod_{i \in I} \prod_{a=1}^{n_i} z_{ia}^{r_i}, \label{eqn:shift plus} \\ 
&\sigma_{\br} : \CV^{\op} \rightarrow \CV^{\op}, \qquad R(z_{i1},\dots,z_{in_i}) \mapsto R(z_{i1},\dots,z_{in_i}) \prod_{i \in I} \prod_{a=1}^{n_i} z_{ia}^{-r_i} .\label{eqn:shift minus}
\end{align} 

\medskip

\begin{definition}
\label{def:quantum loop algebra} 

If we let
\begin{equation}
\label{eqn:quantum loop algebra quotient}
\UUpm = \tUUpm \Big/ \emph{Ker }\widetilde{\Upsilon}^\pm
\end{equation}
then the quantum loop algebra is defined as
\begin{equation}
\label{eqn:quantum loop algebra}
\UU = \UUp \otimes  \BC \left [\ph_{i,1}^\pm, \ph_{i,2}^\pm, \dots, \kappa_{\ba} \right ]_{i \in I, \ba \in \fh} \otimes \UUm
\end{equation}
with the relations induced by \eqref{eqn:rel 0 loop}-\eqref{eqn:rel 3 loop}. Note that $\UU$ inherits the shifts \eqref{eqn:shift pre quantum}.

\end{definition}

\medskip

\begin{example} 
\label{ex:finite type algebra}

For $\fg$ of finite type, it was proved in \cite{NT} that \eqref{eqn:quantum loop algebra quotient} coincides with the original quantum loop algebra defined by Drinfeld in \cite{Dr} and studied by Beck in \cite{B} and Damiani in \cite{Da}. In other words, a system of generators for the ideal $\emph{Ker }\widetilde{\Upsilon}^+$ is given by the Drinfeld-Serre relations
\begin{equation}
\label{eqn:drinfeld-serre}
\sum_{k=0}^{1-c_{ij}} (-1)^k {1-c_{ij} \choose k}_q \emph{Sym} \left[ e_i(z_1) \dots e_i(z_k) e_j(w) e_i(z_{k+1}) \dots e_i(z_{1-c_{ij}}) \right] = 0
\end{equation}
for all $i \neq j$, and analogously for $\emph{Ker }\widetilde{\Upsilon}^-$ using the $f$ currents. Above, $\emph{Sym}$ denotes symmetrization with respect to the variables $z_1,\dots,z_{1-c_{ij}}$.

\end{example}

\medskip 

\begin{example}
\label{ex:almost algebra}

More generally, if
$$
d_{ij} \in \{0, - \max(d_i,d_j)\}
$$
for all $i \neq j$, then we will say $\fg$ is a \emph{strongly symmetrizable} Kac-Moody Lie algebra. In this case, it follows from \cite{N Arbitrary, N New} that $\emph{Ker }\widetilde{\Upsilon}^+$ is still generated by relations \eqref{eqn:drinfeld-serre}, and so 
$$
\UU = \tUU \Big/ (\text{Drinfeld-Serre relations \eqref{eqn:drinfeld-serre}})
$$
    
\end{example}

\medskip

\begin{example} 
\label{ex:simply laced algebra}

For any simply-laced Kac-Moody Lie algebra $\fg$ (i.e. $d_{ii} = 2, \forall i \in I$), it was shown in \cite{N Loop} that $\emph{Ker }\widetilde{\Upsilon}^+$ is generated by relations \eqref{eqn:rho} below. To set up these relations, consider for any $i \neq j \in I$ and any arithmetic progressions $s,\dots,t$ and $s',\dots, t'$ with ratio 2, such that $s+t=s'+t'$, the following oriented graph

\begin{figure}[h]
\label{fig}
\centering
\includegraphics[scale=0.25]{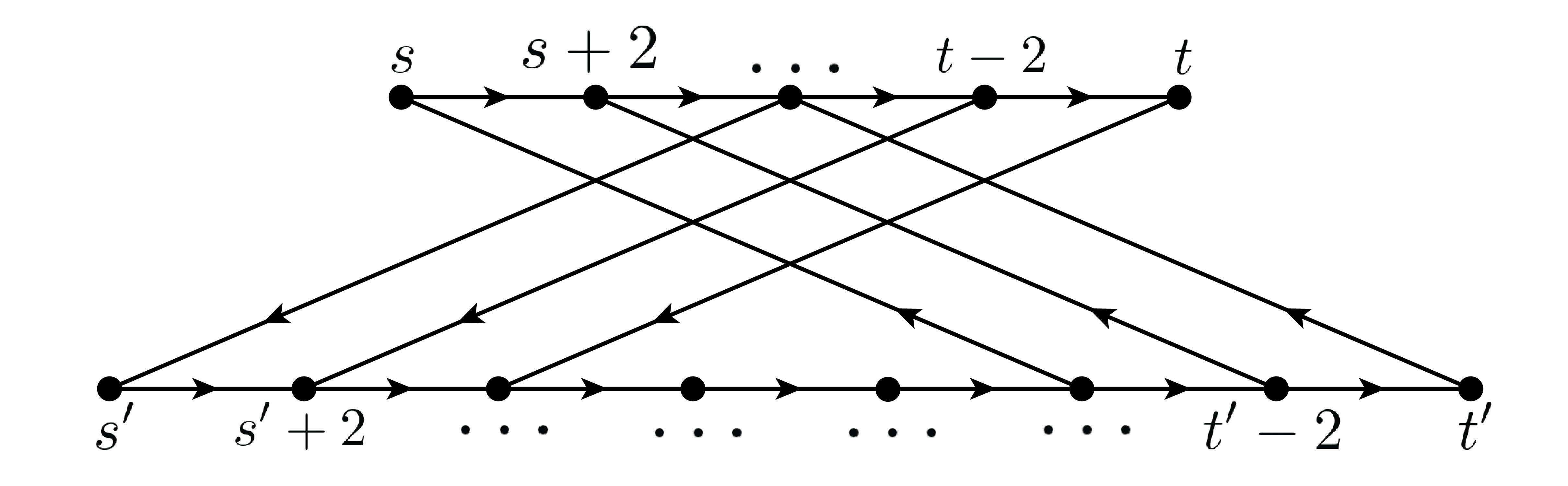} 
\caption{A distinguished zig-zag $Z$}
\end{figure}

\noindent with the diagonal edges pointing $-d_{ij}$ steps to the left. Then we impose the relation
\begin{equation}
\label{eqn:rho}
\sum \Big(\text{a certain polynomial} \Big)  \mathop{\prod_{\text{vertices } c \text{ in}}}_{\text{descending order}} \Big( e_{i}(z_c) \text{ or } e_{j}(z_c) \Big)
\end{equation}
with the sum going over all subgraphs of Figure 1 with no oriented cycles (any such subgraph gives rise to a well-defined order on the set of vertices, with respect to which we can take the product in \eqref{eqn:rho}) and we use either the current $e_i(z_c)$ or $e_j(z_c)$ depending on whether the vertex $c$ lies on the top or bottom row of Figure 1. We refer the interested reader to \cite{N Loop} for the precise polynomials in \eqref{eqn:rho}.

\medskip

\end{example}

\subsection{Hopf algebras}
\label{sub:hopf}

Let us now recall the reason why \eqref{eqn:quantum loop algebra} is a well-defined algebra with respect to relations \eqref{eqn:rel 0 loop}-\eqref{eqn:rel 3 loop}. The reason is the well-known fact that
\begin{align} 
&\tUUg = \tUUp \otimes \BC \left [\ph_{i,1}^+, \ph_{i,2}^+, \dots, \kappa_{\ba} \right ]_{i \in I, \ba \in \fh} \label{eqn:al plus} \\
&\tUUl = \BC \left [\ph_{i,1}^-, \ph_{i,2}^-, \dots, \kappa_{\ba} \right ]_{i \in I, \ba \in \fh} \otimes \tUUm \label{eqn:al minus}
\end{align}
(made into algebras using relations \eqref{eqn:rel 0 loop}-\eqref{eqn:rel 2 loop} and their opposites, respectively) are actually topological Hopf algebras with respect to the Drinfeld coproduct with values in a topological completion of the tensor square:
\begin{equation}
\label{eqn:coproduct ph}
\Delta(\ph^\pm_i(z)) = \ph^\pm_i(z) \otimes \ph^\pm_i(z), \quad \Delta(\kappa_{\ba}) = \kappa_{\ba} \otimes \kappa_{\ba}
\end{equation}
\begin{equation}
\label{eqn:coproduct e}
\Delta(e_i(z)) = \ph_i^+(z) \otimes e_i(z) + e_i(z) \otimes 1
\end{equation}
\begin{equation}
\label{eqn:coproduct f}
\Delta(f_i(z)) = 1 \otimes f_i(z) + f_i(z) \otimes \ph_i^-(z)
\end{equation}
for all $i \in I, \ba \in \fh$, and antipode $S$ given by
\begin{equation}
\label{eqn:antipode phi}
S\left(\ph_i^\pm(z) \right) = \left(\ph^\pm_i(z)\right)^{-1}, \quad S(\kappa_{\ba}) = \kappa_{\ba}^{-1}
\end{equation}
\begin{equation}
\label{eqn:antipode e}
S\left(e_i(z) \right) = -\left(\ph^+_i(z)\right)^{-1} e_i(z)
\end{equation}
\begin{equation}
\label{eqn:antipode f}
S\left(f_i(z) \right) = - f_i(z) \left(\ph^-_i(z)\right)^{-1}.
\end{equation}
for all $i \in I, \ba \in \fh$. Similarly, we can upgrade $\CV$ and $\CV^{\op}$ to topological Hopf algebras by appropriately enlarging them with elements $\ph_{i,d}^+$ and $\ph_{i,d}^-$ respectively. Explicitly, we have the following coproduct formulas, see \cite[Subsection 3.7]{N Cat}:
\begin{align}
\Delta(E) = \sum_{\b0 \leq \bm \leq \bn} \frac {\prod^{j \in I}_{m_j < b \leq n_j} \ph^+_j(z_{jb}) E(z_{i1},\dots , z_{im_i} \otimes z_{i,m_i+1}, \dots, z_{in_i})}{\prod^{i \in I}_{1\leq a \leq m_i} \prod^{j \in I}_{m_j < b \leq n_j} \zeta_{ji} \left( \frac {z_{jb}}{z_{ia}} \right)} \label{eqn:coproduct shuffle plus} \\
\Delta(F) = \sum_{\b0 \leq \bm \leq \bn} \frac {F(z_{i1},\dots , z_{im_i} \otimes z_{i,m_i+1}, \dots, z_{in_i}) \prod^{j \in I}_{1 \leq b \leq m_j} \ph^-_j(z_{jb})}{\prod^{i \in I}_{1\leq a \leq m_i} \prod^{j \in I}_{m_j < b \leq n_j} \zeta_{ij} \left( \frac {z_{ia}}{z_{jb}} \right)} \label{eqn:coproduct shuffle minus} 
\end{align}
for all $E \in \CS_{\bn}$, $F \in \CS_{-\bn}$. To make sense of the right-hand side of formulas \eqref{eqn:coproduct shuffle plus} and \eqref{eqn:coproduct shuffle minus}, we expand the denominator as a power series in the range $|z_{ia}| \ll |z_{jb}|$, and place all the powers of $z_{ia}$ to the left of the $\otimes$ sign and all the powers of $z_{jb}$ to the right of the $\otimes$ sign (for all $i,j \in I$, $1 \leq a \leq m_i$, $m_j < b \leq n_j$). The homomorphisms $\widetilde{\Upsilon}^\pm$ respect the Hopf algebra structures above, so
\begin{align*} 
&\UUg = \UUp \otimes \BC \left [\ph_{i,1}^+, \ph_{i,2}^+, \dots, \kappa_{\ba} \right ]_{i \in I, \ba \in \fh} \\
&\UUl = \BC \left [\ph_{i,1}^-, \ph_{i,2}^-, \dots, \kappa_{\ba} \right ]_{i \in I, \ba \in \fh} \otimes \UUm 
\end{align*}
inherit Hopf algebra structures. Moreover, there exists a Hopf pairing
\begin{equation}
\label{eqn:pairing}
\UUg \otimes \UUl \xrightarrow{\langle \cdot,\cdot\rangle} \BC
\end{equation} 
generated by the assignments
\begin{equation}
\label{eqn:pairing ef}
\Big \langle e_i(x), f_j(y) \Big \rangle = \frac {\delta_{ij}\delta \left(\frac xy \right)}{q_i^{-1} - q_i} 
\end{equation} 
\begin{equation} 
\label{eqn:pairing ph}
\Big \langle \ph^+_i(x), \ph^-_j(y) \Big \rangle =  \frac {xq^{d_{ij}} - y}{x - yq^{d_{ij}}}, \qquad \langle \kappa_{\ba}, \kappa_{\bb} \rangle = q^{-(\ba,\bb)}
\end{equation} 
\footnote{Note that our $q$ and $\kappa_i$ are the usual $q^{-1}$ and $K_i^{-1}$ from the theory of quantum groups.} under the following conditions for all $a,a_1,a_2 \in \UUg$ and $b,b_1,b_2 \in \UUl$
\begin{align}
&\Big \langle a,b_1b_2 \Big \rangle = \Big \langle \Delta(a), b_1 \otimes b_2 \Big \rangle \label{eqn:bialgebra 1} \\
&\Big \langle a_1 a_2 ,b \Big \rangle = \Big \langle a_1 \otimes a_2, \Delta^{\text{op}}(b) \Big \rangle \label{eqn:bialgebra 2} 
\end{align}
($\Delta^{\text{op}}$ is the coproduct opposite to $\Delta$) and
\begin{equation}
\label{eqn:antipode pairing}
\Big \langle S(a), S(b) \Big \rangle = \Big \langle a,b \Big \rangle.
\end{equation}
As shown in \cite{N Arbitrary}, we have that 
\begin{equation}
\label{eqn:drinfeld double}
\UU \text{ is the Drinfeld double } \frac {\UUg \otimes \UUl}{\left( \kappa_{\ba} \otimes 1 - 1 \otimes \kappa_{\ba} \right)_{\ba \in \fh}}.
\end{equation}
Above, recall that the multiplication in the Drinfeld double is controlled by 
\begin{equation}
\label{eqn:drinfeld double relation}
\begin{split}
ba = \Big \langle a_1, S(b_1) \Big \rangle a_2b_2 \Big \langle a_3, b_3 \Big \rangle \\ 
\Leftrightarrow \quad ab = \Big \langle a_1, b_1 \Big \rangle b_2a_2 \Big \langle a_3, S(b_3) \Big \rangle
\end{split}
\end{equation}
for all $a \in \UUg$ and $b \in \UUl$, where we use Sweedler notation 
$$
\Delta^{(2)}(x) = (\Delta \otimes \text{Id} ) \circ \Delta(x) = x_1 \otimes x_2 \otimes x_3,
$$
to avoid writing down the implied summation signs.

\medskip 

\begin{remark}

The above point of view on $\UU$ affords significant technical advantages. For example, showing that the Drinfled new coproduct \eqref{eqn:coproduct ph}, \eqref{eqn:coproduct e}, \eqref{eqn:coproduct f} respects the Drinfeld-Serre relations \eqref{eqn:drinfeld-serre} is a challenging computation directly (see for instance \cite{Da2}). However, the fact that the composition of homomorphisms
$$
\widetilde{\Upsilon}^+ : \tUUp \twoheadrightarrow \UUp \hookrightarrow \CV 
$$
respects the coproduct allows us to conclude that $\Delta$ descends from $\tUU$ to $\UU$.

\end{remark}

\medskip

\subsection{The (small) shuffle algebra}
\label{sub:small shuffle}

We will refer to either of
\begin{equation}
\label{eqn:spherical}
\CS^\pm = \text{Im }\widetilde{\Upsilon}^\pm
\end{equation}
as ``the" shuffle algebra, i.e. the subalgebra of either $\CV$ or $\CV^{\op}$ generated by $\{z_{i1}^d \in \CV_{\bs^i}\}_{i \in I, d\in \BZ}$. By construction, we have
\begin{equation}
\label{eqn:iso shuffle}
\UUpm \cong \CS^{\pm}
\end{equation}
Explicitly, $\CS^-$ is the $\BC$-linear span of rational functions of the form
\begin{equation}
\label{eqn:generators}
\text{Sym} \left[ z_1^{d_1} \dots z_n^{d_n} \prod_{1 \leq a < b \leq n}\zeta_{i_bi_a} \left( \frac {z_b}{z_a} \right) \right]
\end{equation}
as $i_1,\dots,i_n$ run over $I$ and $d_1,\dots,d_n$ run over $\BZ$ (in the expression above, we identify each variable $z_a$ with some variable of the form $z_{i_a\bullet_a}$ in the notation of \eqref{eqn:big shuffle}; the values of $\bullet_a \in \{1,2,\dots\}$ do not matter due to the presence of the symmetrization, as long as we have $\bullet_a \neq \bullet_b$ whenever $a \neq b$, $i_a = i_b$). $\CS^+$ has an analogous description to \eqref{eqn:generators}, but with the product going over $a>b$. The shuffle algebras $\CS^\pm$ are graded by $\pm \nn$, with
$$
\CS_{\bn} = \CS^+ \cap \CV_{\bn} \qquad \text{and} \qquad \CS_{-\bn} = \CS^- \cap \CV^{\op}_{\bn} 
$$
and they inherit the shift automorphisms $\sigma_{\br} : \CS^\pm \rightarrow \CS^\pm$ from \eqref{eqn:shift plus}-\eqref{eqn:shift minus}. Moreover,
\begin{align*} 
&\CS^\geq = \CS^+ \otimes \BC \left [\ph_{i,1}^+, \ph_{i,2}^+, \dots, \kappa_{\ba} \right ]_{i \in I, \ba \in \fh} \\
&\CS^\leq = \BC \left [\ph_{i,1}^-, \ph_{i,2}^-, \dots, \kappa_{\ba} \right ]_{i \in I, \ba \in \fh} \otimes \CS^-
\end{align*}
inherit topological Hopf algebra structures from extended $\CV$ and $\CV^{\op}$ (see \eqref{eqn:coproduct shuffle plus}-\eqref{eqn:coproduct shuffle minus}). Meanwhile, the pairing \eqref{eqn:pairing} takes the following form under the isomorphisms \eqref{eqn:iso shuffle}:
\begin{equation}
\label{eqn:pairing shuffle}
\CS^\geq \otimes \CS^\leq \xrightarrow{\langle \cdot, \cdot \rangle} \BC
\end{equation}
by formula \eqref{eqn:pairing ph} together with \footnote{We will abuse notation in our formulas for the pairing by writing $e_{i,d}$ instead of $z_{i1}^d$, see \eqref{eqn:upsilon plus}.}
\begin{equation}
\label{eqn:pairing shuffle formula}
\Big \langle e_{i_1,d_1} * \dots * e_{i_n,d_n}, F \Big \rangle = \int_{|z_1| \gg \dots \gg |z_n|} \frac {z_1^{d_1} \dots z_n^{d_n} F(z_1,\dots,z_n)}{\prod_{1\leq a < b \leq n} \zeta_{i_bi_a} \left(\frac {z_b}{z_a} \right)} 
\end{equation}
for all $F \in \CS_{-\bn}$, any $d_1,\dots,d_n \in \BZ$ and any $i_1,\dots,i_n \in I$ such that $\bs^{i_1}+\dots+\bs^{i_n} = \bn$ (the notation in the right-hand side of \eqref{eqn:pairing shuffle formula} is defined in accordance with \eqref{eqn:generators}). The subscript under the integral sign means that the variables run over concentric circles centered at the origin, which are very far away from each other and ordered as prescribed by the $\gg$ signs. The volume form $\prod_{a=1}^n \frac {dz_a}{2\pi i z_a}$ will be implied in all our integrals, although we will not write it out explicitly.

\medskip

\begin{example} 
\label{ex:finite type}

For $\fg$ of finite type, we have the following explicit description
\begin{equation}
\label{eqn:e}
\CS^+ = \left\{ \frac {\rho \text{ satisfying \eqref{eqn:wheel 1}}}{\prod^{\text{unordered}}_{i \neq j} \prod_{a, b} (z_{ia} - z_{jb})}  \right\} 
\end{equation}
where $\rho$ goes over color-symmetric Laurent polynomials which satisfy the so-called Feigin-Odesskii wheel conditions for all $i \neq j$ in $I$
\begin{equation}
\label{eqn:wheel 1}
\rho (\dots, z_{ia}, \dots,z_{jb},\dots)\Big|_{(z_{i1},z_{i2}, \dots, z_{i,1-c_{ij}}) \mapsto (z_{j1} q^{d_{ij}}, z_{j1} q^{d_{ij}+d_{ii}},  \dots, z_{j1} q^{-d_{ij}})} =  0
\end{equation}
The inclusion $\subseteq$ of \eqref{eqn:e} was established in \cite{E} based on the seminal work \cite{FO}, while the inclusion $\supseteq$ of \eqref{eqn:e} was proved in \cite{NT}. Compare with Example \ref{ex:finite type algebra}.

\end{example}

\medskip 

\begin{example} 
\label{ex:almost}

For $\fg$ strongly symmetrizable (see Example \ref{ex:almost algebra}), formulas \eqref{eqn:e} and \eqref{eqn:wheel 1} also hold as stated, as proved in \cite[Lemma 3.23]{N New}.

\end{example} 

\medskip

\begin{example} 
\label{ex:simply laced}

For any simply-laced symmetrizable Kac-Moody Lie algebra $\fg$ (i.e. $d_{ii} = 2, \forall i \in I$), it was shown in \cite{N Loop} that 
\begin{equation}
\label{eqn:ee}
\CS^+ = \left\{ \frac {\rho \text{ satisfying \eqref{eqn:wheel 2}}}{\prod^{\text{unordered}}_{i \neq j} \prod_{a, b} (z_{ia} - z_{jb})}  \right\} 
\end{equation}
where $\rho$ goes over color-symmetric Laurent polynomials that satisfy the following condition: for any $i \neq j$ in $I$ and any $t-s, t'-s' \in 2 \BN$ such that $s-t' = s' -t \equiv d_{ij}$ mod 2, 
\begin{equation}
\label{eqn:wheel 2}
(x-y)^{\frac {t-s'+d_{ij}}2+1}
\end{equation}
divides the specialization of $\rho$ at 
\begin{align}
&z_{i1} = xq^s, \ \quad z_{i2} = xq^{s+2}, \quad \ \dots \ \quad z_{i,\frac {t-s}2} = xq^{t-2}, \ \ \quad z_{i, \frac {t-s}2+1} = xq^t \label{eqn:spec 1 intro} \\
&z_{j1} = yq^{s'}, \quad z_{j2} = yq^{s'+2}, \quad \dots \quad z_{j,\frac {t'-s'}2} = yq^{t'-2}, \quad z_{j, \frac {t'-s'}2+1} = yq^{t'} \label{eqn:spec 2 intro}
\end{align}
Compare with Example \ref{ex:simply laced algebra}.

\end{example}

\medskip

\subsection{Slope subalgebras and category $\CO$}
\label{sub:slopes}

 The following are particular cases of slope subalgebras, which were studied in \cite{N Cat}. Let
\begin{align*}
&\CS_{\geq \b0|\bn} = \left\{ E \in \CS_{\bn}  \Big| \lim_{\xi \rightarrow 0} E(\xi z_{i1},\dots,\xi z_{im_i},z_{i,m_{i+1}},\dots,z_{in_i}) < \infty, \forall \b0 < \bm \leq \bn \right\} \\
&\CS_{< \b0|-\bn} = \left\{ F \in \CS_{-\bn}  \Big| \lim_{\xi \rightarrow 0} F(\xi z_{i1},\dots,\xi z_{im_i},z_{i,m_{i+1}},\dots,z_{in_i}) = 0, \forall \b0 < \bm \leq \bn \right\}
\end{align*}
It is straightforward to show that
\begin{align*}
&\CS_{\geq \b0}^+ = \bigoplus_{\bn \in \nn} \CS_{\geq \b0|\bn} \\
&\CS_{< \b0}^- = \bigoplus_{\bn \in \nn} \CS_{< \b0|-\bn}
\end{align*}
are subalgebras of $\CS^+$ and $\CS^-$ (respectively) with respect to the shuffle product. 

\medskip

\begin{proposition}
\label{prop:borel}

The subspace
$$
\CA^{\geq} = \CS_{\geq \b0}^+ \otimes \BC \left [\ph_{i,1}^\pm, \ph_{i,2}^\pm, \dots, \kappa_{\ba} \right ]_{i \in I, \ba \in \fh} \otimes  \CS_{< \b0}^-
$$
of $\UU$ is a subalgebra. If $\fg$ is of finite type, the isomorphism
$$
\UUaff \xrightarrow{\sim} \UU
$$
sends the Borel subalgebra $\UUaff$ isomorphically onto $\CA^{\geq}$.

\end{proposition}

\medskip

\begin{remark} As explained in \cite{N Char}, there exists a factorization
$$
\CA^{\geq} = \prod_{\mu \in (-\infty, \infty]}^{\rightarrow} \CB_{\mu}
$$
where $\CB_{\mu} \subset \CA^{\geq}$ is a subalgebra whose elements all satisfy the identity
$$
\mu(\emph{vdeg } X) = | \emph{hdeg } X | 
$$
(thus, $\mu = \infty$ contains the generators $e_{i,0} \in \CS^+$, while $\mu = 0$ corresponds to the positive half of the loop-Cartan subalgebra). The graded dimension of each $\CB_{\mu}$ is determined by the exponents in the conjectural \cite[formula (90)]{N Char}, which are known for finite type $\fg$. Meanwhile, the commutation relations between general elements in different algebras $\CB_{\mu}$ is not known even in finite types, see \cite{BS} for a classic situation that is closely related to the $\fg = \hgl_1$ analogue of the present paper.
    
\end{remark}

\medskip

 The construction above justifies the following generalization of Definition \ref{def:category o affine}, which applies to all symmetrizable Kac-Moody Lie algebras $\fg$.

\medskip

\begin{definition} 
\label{def:category o borel}

(\cite{N Cat}) Consider the category $\CO$ of complex representations 
\begin{equation}
\label{eqn:rep loop}
\CA^{\geq} \curvearrowright V
\end{equation}
which admit a decomposition 
\begin{equation}
\label{eqn:weight decomposition km}
V = \bigoplus_{\bom \in \cup_{s=1}^t (\bom^s - \nn) } V_{\bom}
\end{equation}
for finitely many $\bom^1,\dots,\bom^t \in \fh$, such that every weight space
\begin{equation}
\label{eqn:weight km}
V_{\bom} = \left\{v \in V \Big| \kappa_{\ba} \cdot v = q^{(\bom, \ba)}v, \ \forall \ba \in \fh \right\}
\end{equation}
is finite-dimensional. If $v \in V_\bom$ as above, then we call $\bom$ the weight of $v$.
\end{definition}

\medskip

\subsection{Simple modules}
\label{sub:simple}

For any rational $\ell$-weight $\bpsi$, we will write $\br = \ord \bpsi \in \zz$ for the tuple of order of poles of $\bpsi = (\psi_i(z))_{i \in I}$ at $z=0$. Consider the following $-\nn$ graded vector spaces, which were introduced in \cite{N Cat, N Char}
\begin{align*}
L(\bpsi) &= \CS_{<\b0}^- \Big/ J(\bpsi) \\
L^{\neq 0}(\bpsi) &= \CS_{<\b0}^- \Big/ J^{\neq 0}(\bpsi) \\
L^{\br} &= \CS_{<\b0}^- \Big/ J^{\br}
\end{align*}
In the formulas above, we define 
$$
\begin{aligned}
J(\bpsi) &= \bigoplus_{\bn \in \nn} J(\bpsi)_{\bn}  \\ 
J^{\neq 0}(\bpsi) &= \bigoplus_{\bn \in \nn} J^{\neq 0}(\bpsi)_{\bn} \\ 
J^{\br} &= \bigoplus_{\bn \in \nn} J^{\br}_{\bn} 
\end{aligned} \qquad \subseteq \qquad \bigoplus_{\bn \in \nn} \CS_{<\b0|-\bn} = \CS^-_{<\b0}
$$
that consist of those elements $F = F(z_{i1},\dots,z_{in_i})_{i \in I} \in \CS_{<\b0|-\bn}$ such that 
\begin{align}
&\left \langle E(z_{i1},\dots,z_{in_i}) \prod_{i \in I} \prod_{a=1}^{n_i} \psi_i(z_{ia}) , S (F) \right \rangle = 0, \quad \forall E \in \CS_{\geq \b0|\bn} \label{eqn:for 1} \\
&\left \langle E(z_{i1},\dots,z_{in_i}) \prod_{i \in I} \prod_{a=1}^{n_i} \psi_i(z_{ia}) , S (F) \right \rangle = 0, \quad \forall E \in \sigma_{N \bone} (\CS_{\geq \b0|\bn}), N \gg 0 \label{eqn:for 2} \\
&\left \langle E(z_{i1},\dots,z_{in_i}) \prod_{i \in I} \prod_{a=1}^{n_i} z_{ia}^{-r_i} , S (F) \right \rangle = 0, \qquad \forall E \in \CS_{\geq \b0|\bn} \label{eqn:for 3}
\end{align}
respectively (in the right-hand side, $S$ denotes the antipode map that we will not need to review). One of the main results of \cite{N Cat} is that there is an action
\begin{equation}
\label{eqn:simple module Borel}
\CA^{\geq} \curvearrowright L(\bpsi)
\end{equation}
with respect to which the RHS is the unique (up to isomorphism) simple $\CA^{\geq}$-module generated by a single vector $\vac$ that satisfies the relations 
\begin{equation}
\label{eqn:simple module relations 1}
\ph_{i,d}^+ \cdot \vac = \psi_{i,d} \vac
\end{equation}
\begin{equation}
\label{eqn:simple module relations 2}
e_{i,d}\cdot \vac = 0 
\end{equation}
for all $i \in I$ and $d \geq 0$.

\medskip

\begin{remark}
\label{rem:enlarge simple}

If we work with the enlarged Cartan subalgebra as in Remark \ref{rem:enlarge cartan}, then we actually obtain a simple module for all
$$
(\bpsi,\bom)
$$
where $\bom \in \fh$ has the property that $\psi_{i,0} = q^{(\bom,\bs^i)}$ for all $i \in I$, as in \cite{HTG}. Such a simple module also satisfies the relation
\begin{equation}
\label{eqn:simple module relations 3}
\kappa_{\ba} \cdot \vac = q^{(\bom,\ba)} \vac
\end{equation}
for all $\ba \in \fh$, on top of \eqref{eqn:simple module relations 1} and \eqref{eqn:simple module relations 2}. To keep things readable, we will not include $\bom$ in our notation, and still refer to the simple modules merely as $L(\bpsi)$. 
    
\end{remark}

\medskip

\subsection{Decomposing simple modules}

It was argued in \cite{N Cat} that \eqref{eqn:simple module Borel} is an analogue of the classical construction of irreducible $\fg$-representations as quotients of $U(\fn)$ by the kernel of the Shapovalov form: $\CS^-_{<\b0}$ takes the role of $U(\fn)$ and $J(\bpsi)$ takes the role of kernel of a contravariant form. Moreover, the following decomposition of vector spaces was proved in \cite[Proposition 4.8]{N Char}
\begin{equation}
\label{eqn:l factor}
L(\bpsi) \cong L^{\br} \otimes L^{\neq 0}(\bpsi)
\end{equation}
The isomorphism above was shown to respect the action of the positive loop Cartan subalgebra (generated by the $\ph_{i,d}^+$), and so it implies the following product formula for $q$-characters (which experts have long known in finite types)
\begin{equation}
\label{eqn:chi factor}
\chi_q(L(\bpsi)) = \chi^{\br} \cdot \chi_q(L^{\neq 0}(\bpsi))
\end{equation}
The factor $\chi^{\br}$ in the right-hand side of \eqref{eqn:chi factor} is an ordinary character (that is, a combination of rational $\ell$-weights that are constant), which was calculated in \cite{N Char} in accordance to a conjecture of \cite{MY, W1}. In the present paper, we will mostly focus on the second factor, or equivalently, on the second tensor factor of \eqref{eqn:l factor}. To start with, the following simple statement was proved in \cite[Remark 4.5]{N Char}.

\medskip

\begin{lemma}
\label{lem:iso}

The inclusion $\CS^-_{<\b0} \subset \CS^-$ induces an isomorphism
\begin{equation}
\label{eqn:iso}
L^{\neq 0}(\bpsi) \cong \CS^- \Big/ \bar{J}^{\neq 0}(\bpsi)
\end{equation}
where we write $\bar{J}^{\neq 0}(\bpsi) = \{F \in \CS^- \text{ satisfying \eqref{eqn:for 2}} \}$.

\end{lemma}

\medskip

\subsection{Residues}
\label{sub:residues}

We wish to describe $L^{\neq 0}(\bpsi)$ as a vector space. By Lemma \ref{lem:iso}, this vector space is as explicit as are $\CS^-$ (see \eqref{eqn:generators} and Examples \ref{ex:finite type}, \ref{ex:almost} and \ref{ex:simply laced}) and $\bar{J}^{\neq 0}(\bpsi)$. To describe the latter vector space, we will call 
$$
i_1,\dots,i_n \in I \quad \text{an ordering of} \quad \bn \in \nn
$$
if $\bs^{i_1}+\dots+\bs^{i_n} = \bn$, see \eqref{eqn:generators}. In this case, for any $F \in \CS_{-\bn}$ we will use
$$
F(z_1,\dots,z_n)
$$
to mean the fact that each variable $z_a$ is plugged into one of the variables $z_{i_a\bullet_a}$ of $F$ (the choice of $\bullet_1,\dots,\bullet_n \in \{1,2,\dots\}$ does not matter due to the color-symmetry of $F$, as long as $\bullet_a \neq \bullet_b$ whenever $a \neq b, i_a = i_b$). As proved in \cite{N Cat, N Char}, we have
\begin{equation}
\label{eqn:condition borel}
F \in \bar{J}^{\neq 0}(\bpsi)_{\bn}  \quad \Leftrightarrow  \quad \int_{z_1} \dots \int_{z_n} \frac {z_1^{d_1} \dots z_n^{d_n} F (z_1,\dots,z_n)}{\prod_{1\leq a < b \leq n} \zeta_{i_bi_a} \left(\frac {z_b}{z_a} \right)} \prod_{a=1}^n \psi_{i_a}(z_a) = 0
\end{equation}
for all orderings $i_1,\dots,i_n$ of $\bn$ and all $d_1,\dots,d_n \in \BZ$. In the formula above, the contours of integration run along the difference of two circles, one centered around $\infty$ and one centered around 0 (with $z_n$ much closer to these two singularities than $z_1$). By the residue theorem, we therefore conclude that 
\begin{equation}
\label{eqn:big intersection}
\bar{J}^{\neq 0}(\bpsi)_{\bn} = \bigcap_{\bx \in (\BC^*)^{\bn}} \bar{J}^{\neq 0}(\bpsi)_{\bx}
\end{equation}
where for any
\begin{equation}
\label{eqn:point x}
\bx = (x_{i1},\dots,x_{in_i})_{i \in I} \in (\BC^*)^{\bn} = \prod_{i \in I} (\BC^*)^{n_i}/S_{n_i}
\end{equation}
we define
\begin{equation}
\label{eqn:j x}
\bar{J}^{\neq 0}(\bpsi)_{\bx} = 
\end{equation}
$$
\left\{F \in \CS_{-\bn} \Big| \underset{z_n=x_n}{\text{Res}} \dots \underset{z_1=x_1}{\text{Res}} \frac {z_1^{d_1} \dots z_n^{d_n} F (z_1,\dots,z_n)}{\prod_{1\leq a < b \leq n} \zeta_{i_bi_a} \left(\frac {z_b}{z_a} \right)} \prod_{a=1}^n \psi_{i_a}(z_a) = 0, \ \forall d_1,\dots,d_n \in \BZ \right\}
$$
Above, we let $x_1,\dots,x_n$ denote any ordering of the entries of $\bx$, in accordance with any ordering $i_1,\dots,i_n$ of $\bn$. From the formula above, we see that $F \in \bar{J}^{\neq 0}(\bpsi)_{\bx}$ if and only if finitely many linear combinations of the derivatives of $F$ vanish at the point \eqref{eqn:point x}. Thus, $\CS^- / \bar{J}^{\neq 0}(\bpsi)_{\bx}$ is a finite-dimensional ring which is annihilated by a sufficiently high power of the maximal ideal of the point $\bx \in (\BC^*)^{\bn}$. Since only finitely many points $\bx$ will produce a non-trivial residue in formula \eqref{eqn:j x} (this is because the rational functions $\psi_i$ and $\zeta_{ij}$ can only produce finitely many poles), then we conclude from \eqref{eqn:big intersection} that
\begin{equation}
\label{eqn:decomposition x}
L^{\neq 0}(\bpsi)_{\bn} = \bigoplus_{\bx \in (\BC^*)^{\bn}} L^{\neq 0}(\bpsi)_{\bx}
\end{equation}
where
\begin{equation}
\label{eqn:l x}
L^{\neq 0}(\bpsi)_{\bx} = \CS_{-\bn} \Big/ \bar{J}^{\neq 0}(\bpsi)_{\bx}
\end{equation}
The decomposition \eqref{eqn:decomposition x} is precisely the one by eigenvalues of the loop Cartan subalgebra $\{\ph_{i,0}^+,\ph_{i,1}^+,\dots\}_{i \in I}$, and indeed as shown in \cite{N Cat} we have
\begin{equation}
\label{eqn:decomposition x q character}
\chi_q(L^{\neq 0}(\bpsi)) = [\bpsi] \sum_{\bn \in \nn} \sum_{\bx \in (\BC^*)^{\bn}} \dim_{\BC} \left(L^{\neq 0}(\bpsi)_{\bx}\right) \prod_{i \in I} \prod_{a=1}^{n_i} A_{i,x_{ia}}^{-1}
\end{equation}
where $A_{i,x}$ denotes the well-known $\ell$-weight of \cite{FR, FM}. In our conventions, it is given by the $I$-tuple of rational functions
\begin{equation}
\label{eqn:a weight}
A_{i,x}^{-1} = \left[ \left( \frac {z-xq^{d_{ij}}}{zq^{d_{ij}}-x}\right)_{j \in I} \right]
\end{equation}
Thus, the dimensions of the vector spaces \eqref{eqn:l x} provide the coefficients of the $q$-character of simple modules when expanded in the basis of monomials in the $A_{i,x}^{-1}$'s.

\bigskip

\section{The shifted category $\CO^{\sh}$}
\label{sec:shifted}

 In what follows, the symmetrizable Kac-Moody Lie algebra $\fg$ and $\br \in \zz$ will be arbitrary. We fix $\br = (r_i)_{i \in I}$ as above, and associate to it the integral coweight
\begin{equation}
\label{eqn:integral weight}
\mu = \sum_{i \in I} r_i \omega_i^\vee
\end{equation}
We will consider the shifted quantum loop algebra $\UU^{\mu}$ of \cite{FT}. By analogy with the previous Section, we construct a simple module
$$
\UU^{\mu} \curvearrowright L^{\sh}(\bpsi)
$$
for any $\ell$-weight $\bpsi$ with $\br = \ord \bpsi$. We will show that $L^{\sh}(\bpsi)$ is isomorphic to the same-named simple module in the shifted category $\CO$ defined in \cite{H Shifted}, and then we will prove Theorems \ref{thm:main intro} and \ref{thm:technical intro}. 

\medskip

\subsection{The shifted quantum loop algebra} 
\label{sub:shifted} 

It is well-known that the shifted quantum loop algebra $\UU^{\mu}$ has the same positive, negative and Cartan subalgebras as $\UU$, but differs in the way they are glued. In fact, we can define the shifted quantum loop algebra as a shifted Drinfeld double, in the following sense. We recall that $\br$ and $\mu$ are always related by formula \eqref{eqn:integral weight}.

\medskip

\begin{definition}
\label{def:shifted}

For any symmetrizable Kac-Moody Lie algebra $\fg$ and any $\br \in \zz$, we consider the $\mu$ shifted Drinfeld double
\begin{equation}
	\label{eqn:shifted quantum loop algebra}
	\UUsh = \UUg \otimes \UUl
\end{equation}
where the multiplication on the vector space in the RHS is controlled by the following shifted version of relations \eqref{eqn:drinfeld double relation}
\begin{equation}
\label{eqn:shifted drinfeld double relation}
\begin{split}
ba = \Big \langle \sigma_{-\br}(a_1), S(b_1) \Big \rangle a_2b_2 \Big \langle a_3, b_3 \Big \rangle \quad \Leftrightarrow \\ 
\Leftrightarrow \quad ab = \Big \langle \sigma_{-\br}(a_1), b_1 \Big \rangle b_2a_2 \Big \langle a_3, S(b_3) \Big \rangle
\end{split}
\end{equation}
with $\sigma_{\br}$ being the automorphism $e_{i,d} \mapsto e_{i,d+r_i}, f_{i,d} \mapsto f_{i,d-r_i}, \ph_{i,d}^{\pm} \mapsto  \ph_{i,d}^{\pm}, \kappa^\pm_{\ba} \mapsto \kappa^\pm_{\ba}$.

\end{definition}

\medskip

 We denote the Cartan elements in $\UUg$ and $\UUl$ by $\kappa_{\ba}^+$ and $\kappa_{\ba}^-$, respectively, and do not require them to be set equal to each other as in \eqref{eqn:al plus}-\eqref{eqn:al minus}. The consequence of this convention is that the usual quantum loop algebras are related to shifted quantum loop algebras by
$$
\UU = \UU^0 \Big/ (\kappa_\ba^+ = \kappa_\ba^-)_{\ba \in \fh}
$$
As a consequence of relation \eqref{eqn:shifted drinfeld double relation}, the shifted analogue of relation \eqref{eqn:rel 3 loop} is
\begin{equation}
\label{eqn:rel 3 loop shifted}
\left[ e_i(x), f_j(y) \right] = \frac {\delta_{ij}\delta \left(\frac xy \right)}{q_i - q_i^{-1}}  \Big( \ph_i^+(x) - y^{-r_i} \ph_i^-(y) \Big)
\end{equation}
but relations \eqref{eqn:rel 0 loop}-\eqref{eqn:rel 2 loop} (as well as their analogues when $e$'s are replaced by $f$'s) hold as stated. Thus, we leave it as an exercise to the reader to check that there is an isomorphism
\begin{equation}
\label{eqn:conventions}
\UU^{\mu} \cong \mathcal{U}_{0,\mu} \left(\widehat{\fg} \right)
\end{equation}
where the right-hand side is the algebra defined in \cite[Subsection 3.1]{H Shifted}. Note that our Cartan currents $\ph^\pm_i(z)$ differ from those of \loccit by the coordinate change $z \rightarrow z^{-1}$ and an overall renormalization, which ensures that our $\ph^\pm_i(z)$ start at $z^0$. 

\medskip

\begin{remark} 

In \loccitt, the algebra in the RHS of \eqref{eqn:conventions} was only considered for finite type $\fg$, which is why the Drinfeld-Serre relation \eqref{eqn:drinfeld-serre} was imposed (see \cite[formula (3.6)]{H Shifted}). For a general symmetrizable Kac-Moody Lie algebra $\fg$, this relation must be replaced by a system of generators of $\emph{Ker } \widetilde{\Upsilon}^\pm$, see for example relations \eqref{eqn:rho} for simply-laced $\fg$. 

\end{remark} 

\medskip

 Since the positive and negative halves of $\UU^{\mu}$ are the same as those of $\UU$, we will still use the shuffle algebra realization of Subsection \ref{sub:small shuffle}, and denote the subalgebras generated by $\{e_{i,d}\}$ and $\{f_{i,d}\}$ respectively by
\begin{equation}
\label{eqn:iso shuffle shifted}
\UUpm^{\mu} \cong \CS^{\pm}.
\end{equation}
It is convenient to replace the Cartan elements $\{\ph_{i,d}^\pm\}^{i \in I}_{d \geq 0}$ by $\{\kappa_{\ba}^\pm,p_{i,u}\}^{\ba \in \fh}_{i \in I, u \in \BZ \backslash 0}$ via
\begin{equation}
\label{eqn:k and p}
\ph^\pm_i(x) = \kappa_{\pm \bs^i}^{\pm} \exp \left( \sum_{u=1}^{\infty} \frac {p_{i,\pm u}}{ux^{\pm u}} \right).
\end{equation}
This is because the commutation relations between these new Cartan elements and $\CS^\pm \cong \UUpm$ can be more succintly written as
\begin{equation}
\label{eqn:k shuffle}
\kappa^+_{\ba} X = X \kappa^+_{\ba} q^{(\pm \bn, \ba)}, \quad \kappa^-_{\ba} X = X \kappa^-_{\ba} q^{(\pm \bn, \ba)}
\end{equation}
\begin{equation}
\label{eqn:p shuffle}
\left[p_{i,u}, X \right] = \pm X \sum_{j \in I} \left(z_{j1}^u+\dots+z_{jn_j}^u\right)(q^{ud_{ij}} - q^{-ud_{ij}})
\end{equation}
for any $X(z_{j1},\dots,z_{jn_j})_{j \in I} \in \CS_{\pm \bn}$. As before, these relations take the same form in the shifted and non-shifted cases, and $\mu = \sum_{i \in I} r_i \omega_i^\vee$ only plays a role in \eqref{eqn:rel 3 loop shifted}.

\medskip

\subsection{Shifted category $\CO$ and simple modules}
\label{sub:shifted category o}

For a general symmetrizable Kac-Moody Lie algebra $\fg$, the following analogue of Definition \ref{def:category o borel} was constructed by the first-named author.

\medskip

\begin{definition} 
\label{def:category o shifted}

(\cite{H Shifted}) Consider the category $\CO^{\esh}$ of complex representations 
\begin{equation}
\label{eqn:rep shifted}
\UUsh \curvearrowright V
\end{equation}
which admit a decomposition \eqref{eqn:weight decomposition km} with finite-dimensional weight spaces \eqref{eqn:weight km}.

\end{definition}

\medskip

In \cite{H Shifted}, the first-named author constructed for any rational $\ell$-weight $\bpsi$ of order $\br$ (recall the notation $\mu = \sum_{i \in I} r_i \omega_i^\vee$) a unique simple module
\begin{equation}
\label{eqn:shifted simple}
\UUsh \curvearrowright L^{\sh}(\bpsi)
\end{equation}
generated by a single vector $\vac$ that satisfies 
\begin{equation}
\label{eqn:simple module relations 1 shifted}
\begin{split}
&\ph_i^+(z) \cdot \vac = \psi_i(z) \vac \qquad \text{ expanded near } z \sim \infty \\
&\ph_i^-(z) \cdot \vac = z^{r_i} \psi_i(z) \vac \quad \text{expanded near } z \sim 0
\end{split}\end{equation}
for all $i \in I$, and
\begin{equation}
\label{eqn:simple module relations 2 shifted}
e_{i,d}\cdot \vac = 0 
\end{equation}
for all $i \in I$, $d \in \BZ$.

\begin{theorem} The representation $L^{\esh}(\bpsi)$ is in $\CO^{\esh}$ if and only if $\bpsi$ is rational.
\end{theorem}

The proof of this statement in \cite[Theorem 4.12]{H Shifted} relies on the arguments of \cite[Theorem 4.9]{HTG} and \cite[Section 5]{CP94} for ordinary quantum affinizations, that work both for Cartan matrices of finite type or general type (as explained in \cite{HLMS}). Thus the proof in \cite{H Shifted} holds for general type (it does not involve the relations imposed in \eqref{eqn:quantum loop algebra quotient}, which are not explicitly known in general symmetrizable types).

\medskip

Using Drinfeld's topological coproduct, a fusion product was constructed for simple representations in category $\CO^{\text{sh}}$ (\cite{H Shifted}). The corresponding proofs also extend from finite to general symmetrizable types. 

\medskip

Note that representations in $\CO^{\sh}$ are not necessarily of finite length and do not 
necessarily have the Jordan-H\"older property. However, the multiplicity of a simple module in an object of category $\CO^{\sh}$ is well-defined, and thus the (topological) Grothendieck ring is well-defined. In finite types, the category of finite length modules is stable by fusion product (see \cite{HZ} and Theorem \ref{fls} below), but the analogous statement is not clear for general symmetrizable types.

\medskip

\subsection{Technical results}
\label{sub:technical}

We will now proceed to give a Verma-module like description of the simple module \eqref{eqn:shifted simple}, in order to compare it to the object $L^{\neq 0}(\bpsi)$ considered in the previous Section. We begin with the following analogues of the technical results \cite[Definition 4.3 and Proposition 4.5]{N Cat}, compare with \eqref{eqn:for 1}.

\medskip

\begin{definition}
\label{def:verma module}

For any rational $\ell$-weight $\bpsi$ of order $\br$, consider the representation
\begin{equation}
\UUsh \curvearrowright W^{\esh}(\bpsi)
\end{equation}
generated by a single vector $\vac$ modulo the relations
\begin{equation}
\label{eqn:two phis}
\begin{split}
&\ph_i^+(z) \cdot \vac = \psi_i(z) \vac \qquad \text{ expanded near } z \sim \infty \\
&\ph_i^-(z) \cdot \vac = z^{r_i} \psi_i(z) \vac \quad \text{expanded near } z \sim 0 \\
&\kappa^\pm_{\ba} \cdot \vac = q^{(\bom,\ba)} \vac
\end{split}
\end{equation}
for all $i \in I$, $\ba \in \fh$ (in the third equation, $\bom \in \fh$ is a weight as in Remark \ref{rem:enlarge simple}, even though we do not include it in our notation), and 
\begin{equation}
\label{eqn:e ann}
e_{i,d} \cdot \vac = 0
\end{equation}
for all $i \in I$, $d \in \BZ$.

\end{definition}

\medskip

 By the very nature of the Drinfeld double construction, we have a triangular decomposition (linear isomorphism)
\begin{equation}
\label{eqn:triangular shifted}
\UUsh =  \CS^+ \otimes \BC \left [\ph_{i,1}^\pm, \ph_{i,2}^\pm, \dots, \kappa^\pm_{\ba} \right ]_{i \in I, \ba \in \fh} \otimes \CS^-.
\end{equation}
We have a vector space isomorphism
\begin{equation}
\label{eqn:iso verma}
W^{\sh}(\bpsi) \cong \CS^-
\end{equation}
and so $W^{\sh}(\bpsi)$ inherits the $(-\nn)$-grading from $\CS^-$.

\medskip

\begin{definition}
\label{def:ideal} 

For any $\ell$-weight $\bpsi$ of order $\br$, consider the linear subspace
\begin{equation}
	\label{eqn:graded j}
J^{\esh}(\bpsi) = \bigoplus_{\bn \in \nn}	J(\bpsi)_{\bn} \subseteq  \bigoplus_{\bn \in \nn} \CS_{-\bn} = \CS^-
\end{equation}
consisting of those shuffle elements $F(z_{i1},\dots,z_{in_i})_{i \in I} \in \CS_{-\bn}$ such that 
\begin{equation}
\label{eqn:psi pairing shifted}
\left \langle E(z_{i1},\dots,z_{in_i})_{i \in I} \prod_{i \in I} \prod_{a=1}^{n_i} \psi_i(z_{ia}) , F_1 * S (F_2) \right \rangle = 0, \quad \forall E \in \CS_{\bn}
\end{equation}
where in the right-hand side we write $\Delta(F) = F_1 \otimes F_2$ for the coproduct \eqref{eqn:coproduct shuffle minus} in Sweedler notation \footnote{If we write $F_1(z_{i1},\dots,z_{im_i})_{i \in I}$ and  $F_2(z_{i,m_i+1},\dots,z_{in_i})_{i \in I}$ for the variables that appear in the two tensor factors of $\Delta(F) = F_1 \otimes F_2$, it is very important to keep in mind that the rational functions $\psi_i$ in \eqref{eqn:psi pairing shifted} must be expanded in the range 
\begin{equation}
\label{eqn:expansion}
|z_{i1}|,\dots,|z_{i m_i}| \ll |z_{i,m_i+1}|,\dots,|z_{in_i}|
\end{equation}
This is the reason why \eqref{eqn:psi pairing shifted} does not vanish identically, despite the fact that $F_1 * S(F_2) = 0$ in any topological Hopf algebra due to the properties of the antipode.}. Compare \eqref{eqn:psi pairing shifted} with \eqref{eqn:for 1}.

\end{definition}

\medskip

Let us explain why the pairing in \eqref{eqn:psi pairing shifted} is well-defined for any given $E$ and $F$, even if $F_1 \otimes F_2$ is an infinite sum. Due to the expansion in \eqref{eqn:coproduct shuffle minus}, all but finitely many of the summands $F_1 \otimes F_2$ will have the property that the homogeneous degree of $F_1$ is bounded below and that of $F_2$ is bounded above by any given number. Since $\langle E,F \rangle \neq 0$ only if the homogeneous degrees of $E$ and $F$ add up to 0, the choice of expansion in \eqref{eqn:expansion} implies that $E \prod_{i \in I} \prod_{a=1}^{n_i} \psi_i(z_{ia})$ pairs trivially with all but finitely many of the $F_1 * S(F_2)$. 

\medskip

\begin{proposition}
\label{prop:ideal} 

$J^{\esh}(\bpsi) \vac$ is the unique maximal graded $\UUsh$ submodule of $W^{\esh}(\bpsi)$. 

\end{proposition}

\medskip 

\begin{proof} We need to show that the subspace $J^{\sh}(\bpsi) \vac \subseteq W^{\sh}(\bpsi)$ is preserved by 

\medskip

\begin{enumerate}[leftmargin=*]

\item left multiplication with $\CS^-$,

\medskip

\item left multiplication with $\{\kappa_\ba^\pm,p_{i,u}\}_{\ba \in \fh, i \in I, u \neq 0}$,

\medskip

\item left multiplication with $\CS^+$.

\end{enumerate}

\medskip

 To prove (1), let us consider any $F' \in \CS^-$, $F'' \in  J^{\sh}(\bpsi)$ and $E \in \CS^+$. The fact that the coproduct is a homomorphism, the antipode is an anti-homomorphism and property \eqref{eqn:bialgebra 1} imply that
\begin{equation}
\label{eqn:*}
\begin{split}
& \left \langle E \prod \psi , (F'* F'')_1 * S \left( (F' * F'')_2 \right) \right \rangle \\
=& \left \langle E \prod \psi , F'_1 * F''_1 * S (F_2'') * S(F_2') \right \rangle  \\
=& \left \langle \Delta^{(3)} (E \prod \psi) ,F'_1 \otimes F''_1 \otimes S (F_2'') \otimes S(F_2') \right \rangle
\end{split}
\end{equation}
where $E\prod \psi$ is shorthand for $E(z_{i1},\dots,z_{in_i})_{i \in I} \prod_{i \in I} \prod_{a=1}^{n_i} \psi_i(z_{ia})$. The following Claim is an easy consequence of \eqref{eqn:bialgebra 2} and \eqref{eqn:antipode pairing}, which we leave as an exercise to the reader (see the analogous statement in \cite{N Cat}).

\medskip

\begin{claim}
\label{claim:pairing}

If \eqref{eqn:psi pairing shifted} holds for all $E \in \CS^+$, then it also holds for all $E \in \CS^{\geq}$.

\end{claim}

\medskip

 The claim above and the fact that $F'' \in J^{\sh}(\bpsi)$ imply that the second-and-third of the four tensor factors in the RHS of \eqref{eqn:*} have pairing 0. Thus, the whole pairing in \eqref{eqn:*} is 0, hence $F' * F'' \in J^{\sh}(\bpsi)$, as required.

\medskip

 To prove (2), recall from \eqref{eqn:k shuffle} and \eqref{eqn:p shuffle} that commuting $F$ with $\kappa^\pm_\ba$ and $p_{i,u}$ amounts to multiplying $F$ by either a scalar or a color-symmetric Laurent polynomial $\rho(z_{ia})$. Since the pairing in \eqref{eqn:psi pairing shifted} is a certain contour integral applied to the product of $E$, $F$ and $\psi$, see the general formula \eqref{eqn:pairing shuffle formula}, multiplying $F$ by $\rho$ has the same effect on the pairing as multiplying $E$ by $\rho$. Since $\CS^+$ is preserved under multiplication by color-symmetric Laurent polynomials, (2) follows. 

\medskip

 For statement (3), we invoke \eqref{eqn:shifted drinfeld double relation} for any $E \in \CS_{\bn}$ and $F \in \CS_{-\bn}$:
\begin{equation}
\label{eqn:base}
E * F = \Big \langle \sigma_{-\br}(E_1), F_1 \Big \rangle F_2 * E_2 \Big \langle E_3, S(F_3) \Big \rangle 
\end{equation}
where $\Delta^{(2})(E) = E_1 \otimes E_2 \otimes E_3 \in \CS_{\bn'} \otimes \CS_{\bn''} \otimes \CS_{\bn'''}$. When we apply \eqref{eqn:base} to $\vac \in W^\sh(\bpsi)$, only the $\bn'' = \b0$ terms survive, and indeed we obtain
\begin{equation}
\label{eqn:r on vac}
E \left( F \vac \right) = F_2 \vac \cdot \Big \langle \sigma_{-\br}(E_1), F_1 \Big \rangle \Big \langle E_2 \prod \psi , S(F_3) \Big \rangle 
\end{equation}
The product of $\psi$'s in the second pairing is produced by the action of Cartan elements in the middle tensor factor of $\Delta^{(2)}(E)$ on the vacuum. We must prove that $F \in J^{\sh}(\bpsi)$ implies that the RHS of \eqref{eqn:r on vac} lies in $J^{\sh}(\bpsi) \vac$. Let us unpack the right-hand side of \eqref{eqn:r on vac}. By \eqref{eqn:coproduct shuffle minus}, we have
\begin{equation}
\label{eqn:reason}
\Delta^{(2)}(F) = F'(z_{ia}) \otimes F''(z_{jb}) \prod_{i,a} \ph_i^-(z_{ia}) \otimes F'''(z_{kc}) \prod_{i,a} \ph_i^-(z_{ia}) \prod_{j,b} \ph_j^-(z_{jb}),
\end{equation}
for various $F',F'',F''' \in \CS^-$ whose variables we will denote by $z_{ia}, z_{jb}, z_{kc}$, respectively. With this in mind, formula \eqref{eqn:r on vac} reads
\begin{equation}
\label{eqn:r on vac 2}
E \left( F \vac \right) = 
\end{equation}
\begin{align*} 
&= F'' \prod_{i,a} \ph_i^-(z_{ia}) \vac \cdot \Big \langle \sigma_{-\br}(E_1), F' \Big \rangle \Big \langle E_2 \prod \psi , S(F''' \prod_{i,a} \ph_i^-(z_{ia}) \prod_{j,b} \ph_j^-(z_{jb})) \Big \rangle = \\
&= F'' \vac \cdot \Big \langle E_1 \prod \psi, F' \Big \rangle \Big \langle E_2 \prod \psi , S(F''' \prod_{i,a} \ph_i^-(z_{ia}) \prod_{j,b} \ph_j^-(z_{jb})) \Big \rangle  
\end{align*}
with the equality between the latter two rows following from the middle equation of \eqref{eqn:two phis}. To show that the expression on the bottom row lies in $J^{\sh}(\bpsi) \vac$, we must prove that for all $E' \in \CS^+$ we have \footnote{In the middle pairing in the right-hand side of equation \eqref{eqn:0 want}, we are using the fact that
\begin{multline*}
F_2 * S(F_3) = F''_1 \prod_{i,a} \ph_i^-(z_{ia}) * S\Big(F_2'' \prod_{i,a} \ph_i^-(z_{ia}) \Big) = \\ = F''_1\prod_{i,a} \ph_i^-(z_{ia}) * \prod_{i,a} \ph_i^-(z_{ia})^{-1} S(F_2'') = F''_1 * S(F_2'') 
\end{multline*}
where $F_2 \otimes F_3$ denotes the coproduct $\Delta \left( F'' \prod_{i,a} \ph_i^-(z_{ia}) \right)$ of the middle tensor factor in \eqref{eqn:reason}.}
\begin{equation}
\label{eqn:0 want}
0 =\Big \langle E_1 \prod \psi, F_1 \Big \rangle \Big \langle E' \prod \psi, F_2 * S(F_3) \Big \rangle \Big \langle E_2 \prod \psi , S(F_4) \Big \rangle
\end{equation}
By \eqref{eqn:bialgebra 1}, equality \eqref{eqn:0 want} becomes
\begin{equation*}
\begin{split}
0 =&\Big \langle E_1 \prod \psi, F_1 \Big \rangle \Big \langle E'_1 \prod \psi, F_2 \Big \rangle \Big \langle E_2', S(F_3) \Big \rangle \Big \langle E_2 \prod \psi , S(F_4) \Big \rangle \\
\stackrel{\eqref{eqn:bialgebra 2}}=& \Big \langle E_1' * E_1 \prod \psi, F_1 \Big \rangle \Big \langle E_2' * E_2 \prod \psi, S(F_2) \Big \rangle \stackrel{\eqref{eqn:bialgebra 1}}= \Big \langle E' * E \prod \psi, F_1 * S(F_2) \Big \rangle.
\end{split}
\end{equation*}
The latter is a true equality due to the assumption that $F \in J^{\sh}(\bpsi)$.

\medskip

 Having showed that $J^{\sh}(\bpsi) \vac$ is a graded $\UUsh$ submodule of $W^{\sh}(\bpsi)$, it remains to show that it is the unique such maximal graded submodule. To this end, choose any $F \in \CS_{-\bn} \backslash J^{\sh}(\bpsi)_{\bn}$, which means that there exists $E \in \CS_{\bn}$ such that
$$
\left \langle E(z_{i1},\dots,z_{in_i})_{i \in I} \prod_{i \in I} \prod_{a=1}^{n_i} \psi_i(z_{ia}) , F_1 * S(F_2) \right \rangle =: \alpha \neq 0
$$
Because of \eqref{eqn:reason}, formula \eqref{eqn:r on vac} implies precisely $EF \vac = \alpha \vac$. Therefore, any graded submodule of $W^{\sh}(\bpsi)$ which strictly contains $J^{\sh}(\bpsi) \vac$ must contain the highest weight vector $\vac$, and thus must be the whole of $W^{\sh}(\bpsi)$. 

\end{proof}

\medskip

\begin{corollary}
\label{cor:simple module}

(\textbf{Theorem \ref{thm:technical intro}}) For any $\ell$-weight $\bpsi$ of order $\br$, the quotient
\begin{equation}
\label{eqn:simple module cor}
L^{\esh}(\bpsi) = W^{\esh}(\bpsi) \Big/ J^{\esh}(\bpsi) \vac 
\end{equation}
is the unique (up to isomorphism) simple graded $\UUsh$ module generated by a single vector $\vac$ that satisfies the properties 
\begin{equation}
\label{eqn:simple module shifted property}
\begin{split}
&\ph_i^+(z) \cdot \vac = \psi_i(z) \vac \qquad \text{ expanded near } z \sim \infty \\
&\ph_i^-(z) \cdot \vac = z^{r_i} \psi_i(z) \vac \quad \text{expanded near } z \sim 0 \\
&\kappa^\pm_{\ba} \cdot \vac = q^{(\bom,\ba)} \vac
\end{split}
\end{equation}
for all $i \in I$, $\ba \in \fh$ and 
\begin{equation}
\label{eqn:simple module shifted property 2}
e_{i,d}\cdot \vac = 0
\end{equation}
for all $i \in I$, $d \in \BZ$.

\end{corollary}

\medskip

\subsection{Residues revisited}
\label{eqn:residues revisited}

The following technical claim is an analogue of \cite[Lemma 3.10]{N Cat}, and it will be proved at the end of the present Subsection.

\medskip

\begin{lemma}
\label{lem:antipode}

For any $i_1,\dots,i_n \in I$ and $d_1, \dots,d_n \in \BZ$, we have
\begin{equation}
\label{eqn:antipode}
\left \langle e_{i_1,d_1} * \dots * e_{i_n,d_n} \prod_{i \in I} \prod_{a=1}^{n_i} \psi_i (z_{ia}) , F_1 * S (F_2) \right \rangle = \sum_{m=0}^n (-1)^{n-m} 
\end{equation}
\begin{multline*}
\sum_{\{1,\dots,n\} = \{a_1 < \dots < a_m\} \sqcup \{b_1 < \dots < b_{n-m}\}}  \int_{|z_{a_m}| \ll \dots \ll |z_{a_1}| \ll 1 \ll |z_{b_1}| \ll \dots \ll |z_{b_{n-m}}|}  \\
\frac {z_1^{d_1} \dots z_n^{d_n} F(z_1,\dots,z_n)}{\prod_{1\leq a < b \leq n} \zeta_{i_bi_a} \left(\frac {z_b}{z_a} \right)} \prod_{a=1}^n \psi_{i_a}(z_a)
\end{multline*}
for all $F \in \CS_{-\bs^{i_1} - \dots - \bs^{i_n}}$. On the bottom row, the notation $F(z_1,\dots,z_n)$ refers to plugging each symbol $z_a$ into a variable of the form $z_{i_a\bullet_a}$ of $F$, see \eqref{eqn:pairing shuffle formula}.

\end{lemma}

\medskip

 Before we prove Lemma \ref{lem:antipode}, let us use it to conclude the proof of Theorem \ref{thm:main intro}.

\medskip

\begin{proof} \emph{of Theorem \ref{thm:main intro}:} Since $\CS^+$ is spanned by shuffle products of $e_{i,d} = z_{i1}^d \in \CV_{\bs^i}$, we conclude that $F \in \CS_{-\bn}$ lies in $J^{\sh}(\bpsi)$ if and only if the left-hand side of \eqref{eqn:antipode} is 0 for all $i_1,\dots,i_n \in I$ and $d_1,\dots,d_n \in \BZ$. However, this condition is equivalent to the right-hand side of \eqref{eqn:antipode} being 0, which is tautologically the same as condition \eqref{eqn:condition borel}. This establishes the equality of vector spaces
\begin{equation}
\label{eqn:easy}
\bar{J}^{\neq 0}(\bpsi) = J^{\sh}(\bpsi)
\end{equation}
which together with \eqref{eqn:iso} implies \eqref{eqn:main intro}. The same analysis as in \eqref{eqn:big intersection} implies that
\begin{equation}
\label{eqn:big intersection shifted}
J^{\sh}(\bpsi)_{\bn} = \bigcap_{\bx \in (\BC^*)^{\bn}} J^{\sh}(\bpsi)_{\bx}
\end{equation}
where 
$$
J^{\sh}(\bpsi)_{\bx} = \left\{F \in \CS_{-\bn} \Big| \underset{z_n=x_n}{\text{Res}} \dots \underset{z_1=x_1}{\text{Res}} \frac {F (z_1,\dots,z_n)(\text{any polynomial})}{\prod_{1\leq a < b \leq n} \zeta_{i_bi_a} \left(\frac {z_b}{z_a} \right)} \prod_{a=1}^n \psi_{i_a}(z_a) = 0 \right\}
$$
Comparing the formula above with \eqref{eqn:j x}, we see that 
\begin{equation}
\label{eqn:j sh bar j}
J^{\sh}(\bpsi)_{\bx} = \bar{J}^{\neq 0}(\bpsi)_{\bx}
\end{equation}
for all $\bx \in (\BC^*)^{\bn}$. We conclude the following analogue of \eqref{eqn:decomposition x}
\begin{equation}
\label{eqn:decomposition x shifted}
L^{\sh}(\bpsi)_{\bn} = \bigoplus_{\bx \in (\BC^*)^{\bn}} L^{\sh}(\bpsi)_{\bx}
\end{equation}
where
\begin{equation}
\label{eqn:l x shifted}
L^{\sh}(\bpsi)_{\bx} = \CS_{-\bn} \Big/ J^{\sh}(\bpsi)_{\bx}
\end{equation}
As a consequence of \eqref{eqn:easy}, all the vector spaces with superscripts ``$\sh$" are isomorphic to the corresponding vector spaces with superscript ``$\neq 0$" from Subsection \ref{sub:residues}. Therefore, we conclude that the $q$-characters satisfy
\begin{equation}
\label{eqn:decomposition x q character shifted}
\chi_q(L^{\sh}(\bpsi)) = [\bpsi] \sum_{\bn \in \nn} \sum_{\bx \in (\BC^*)^{\bn}} \dim_{\BC} \left(L^{\sh}(\bpsi)_{\bx}\right) \prod_{i \in I} \prod_{a=1}^{n_i} A_{i,x_{ia}}^{-1} = \chi_q(L^{\neq 0}(\bpsi))
\end{equation}
This establishes \eqref{eqn:main intro q-characters}, precisely as predicted by \cite{H Shifted}.

\end{proof}

\medskip 

\begin{proof} \emph{of Lemma \ref{lem:antipode}:} Let us recall from \eqref{eqn:coproduct shuffle minus} the formula for $\Delta(F) = F_1 \otimes F_2$, and let us abbreviate it as
$$
\Delta(F) = F' \otimes F'' \ph
$$
where $F',F'' \in \CS^-$ and $\ph$ is a polynomial in the $\ph_{j,k}^-$'s. We may use \eqref{eqn:bialgebra 1} to write 
\begin{equation}
\label{eqn:equation 1}
\text{LHS of \eqref{eqn:antipode}} = \left \langle \Delta(e_{i_1,d_1}) \dots \Delta(e_{i_n,d_n}) \prod_{i \in I} \prod_{a=1}^{n_i} \psi_i (z_{ia}) , F' \otimes S (F'' \ph) \right \rangle 
\end{equation}
Formula \eqref{eqn:coproduct e} reads $\Delta(e_{i,d}) = e_{i,d} \otimes 1 + \sum_{k=0}^\infty \ph_{i,k}^+ \otimes e_{i,d-k}$, so \eqref{eqn:equation 1} becomes
\begin{equation}
\label{eqn:equation 2}
\text{LHS of \eqref{eqn:antipode}} = \sum_{\{1,\dots,n\} = \{a_1 < \dots < a_m\} \sqcup \{b_1 < \dots < b_{n-m}\}} \sum_{k_1,\dots,k_{n-m} = 0}^\infty
\end{equation}
$$
\Big \langle (e_{i_1,d_1} \text{ or } \ph_{i_1,k_1}^+) \dots (e_{i_n,d_n} \text{ or } \ph_{i_n,k_n}^+) \prod \psi , F' \Big \rangle \Big \langle \prod_{y=1}^{n-m} e_{i_{b_y},d_{b_y}-k_{b_y}} \prod \psi, S(F'' \ph)  \Big \rangle 
$$
where in each parenthesis $(e_{i_x,d_x} \text{ or } \ph_{i_x,k_x}^+)$ we declare that we choose $e_{i_x,d_x}$ if $x \in \{a_1,\dots,a_m\}$ and $\ph_{i_x,k_x}^+$ if $x \in \{b_1,\dots,b_{n-m}\}$. Using formula \eqref{eqn:rel 1 loop}, we have
\begin{equation}
\label{eqn:phi bar past e 1}
\ph_j^+(y) e_i(x) = e_i(x) \ph_j^+(y) \frac {\zeta_{ji} \left( \frac yx \right)}{\zeta_{ij} \left(\frac xy \right)}  \Rightarrow  \ph_{j,k}^+ e_{i,d} = \sum_{\ell=0}^k \gamma_{ij}^{(\ell)} e_{i,d+\ell} \ph_{j,k-\ell}^+
\end{equation}
where the complex numbers $\gamma_{ij}^{(\ell)}$ are defined by
\begin{equation}
\label{eqn:expansion zeta 1}
\frac {\zeta_{ji} \left( \frac yx \right)}{\zeta_{ij} \left(\frac xy \right)} = \sum_{\ell=0}^{\infty} \gamma_{ij}^{(\ell)} \frac {x^\ell}{y^\ell}
\end{equation}
We can use \eqref{eqn:phi bar past e 1} to move all the $\varphi$'s to the right of all the $e$'s in the second line of relation \eqref{eqn:equation 2}, and we thus obtain
\begin{equation}
\label{eqn:equation 3}
\text{LHS of \eqref{eqn:antipode}} = \sum_{\{1,\dots,n\} = \{a_1 < \dots < a_m\} \sqcup \{b_1 < \dots < b_{n-m}\}} \sum_{k_1,\dots,k_{n-m} = 0}^\infty \sum_{\{\ell_{x,y} \geq 0\}_{\forall a_x> b_y}}
\end{equation}
$$
 \prod_{1 \leq x \leq m, 1 \leq y \leq n-m}^{a_x > b_y} \gamma_{i_{a_x}i_{b_y}}^{(\ell_{x,y})} \left \langle \prod_{x=1}^m e_{i_{a_x}, d_{a_x}+\sum_{1 \leq y \leq n-m}^{a_x>b_y} \ell_{x,y}} \prod_{y=1}^{n-m} \ph_{i_{b_y},k_{b_y}-\sum_{1\leq x \leq m}^{a_x>b_y} \ell_{x,y}}^+\prod \psi , F' \right \rangle 
$$
$$
\left \langle \prod_{y=1}^{n-m} e_{i_{b_y},d_{b_y}-k_{b_y}} \prod \psi, S(F'' \ph)  \right \rangle 
$$
For any $E \in \CS^+$, $F \in \CS^-$ and for any $\ph^+,\ph^-$ polynomials in $\ph^+_{j,k}, \ph^-_{j,k}$ (respectively), we have the following identity
$$
\left \langle E \ph^+, F \ph^- \right \rangle \stackrel{\eqref{eqn:bialgebra 1}}=  \left \langle \Delta(E) \Delta(\ph^+), F \otimes \ph^- \right \rangle \stackrel{\eqref{eqn:coproduct shuffle plus}}= \left \langle (E \otimes 1) (\ph^+_1 \otimes \ph^+_2), F \otimes \ph^- \right \rangle 
$$
$$
= \left \langle E \ph^+_1, F \right \rangle \left \langle \ph^+_2, \ph^- \right \rangle \stackrel{\eqref{eqn:bialgebra 2}}= \left \langle E \otimes \ph^+_1, \Delta^{\text{op}}(F) \right \rangle \left \langle \ph^+_2, \ph^- \right \rangle \stackrel{\eqref{eqn:coproduct shuffle minus}}= \left \langle E \otimes \ph^+_1, F \otimes 1 \right \rangle \left \langle \ph^+_2, \ph^- \right \rangle 
$$
\begin{equation}
\label{eqn:ph in pairing}
= \left \langle E, F \right \rangle \varepsilon(\ph^+_1) \left \langle \ph^+_2, \ph^- \right \rangle =  \left \langle E, F \right \rangle \left \langle \ph^+, \ph^- \right \rangle
\end{equation}
where we write $\Delta(\ph^+) = \ph^+_1 \otimes \ph^+_2$ in Sweedler notation, and let $\varepsilon$ denote the counit. In particular, if $\ph^- = 1$, then the pairing above is zero unless $\ph^+$ is a polynomial in the $\ph^+_{j,0}$'s. Therefore, the second line of \eqref{eqn:equation 3} is 1 if $k_{b_y} = \sum_{1\leq x \leq m}^{a_x>b_y} \ell_{x,y}$ for all $y \in \{1,\dots,n-m\}$ and 0 otherwise. We conclude that \eqref{eqn:equation 3} can be rewritten as 
\begin{equation}
\label{eqn:equation 4}
\text{LHS of \eqref{eqn:antipode}} = \sum_{\{1,\dots,n\} = \{a_1 < \dots < a_m\} \sqcup \{b_1 < \dots < b_{n-m}\}} 
\end{equation}
$$
\sum_{\{\ell_{x,y} \geq 0\}_{\forall a_x> b_y}} \prod_{1 \leq x \leq m, 1 \leq y \leq n-m}^{a_x > b_y} \gamma_{i_{a_x}i_{b_y}}^{(\ell_{x,y})} \left \langle \prod_{x=1}^m e_{i_{a_x}, d_{a_x}+\sum_{1 \leq y \leq n-m}^{a_x>b_y} \ell_{x,y}} \prod \psi , F' \right \rangle 
$$
$$
 \left \langle \prod_{y=1}^{n-m} e_{i_{b_y},d_{b_y}-\sum_{1\leq x \leq m}^{a_x>b_y} \ell_{x,y}} \prod \psi, S(F'' \ph)  \right \rangle 
$$
Using property \eqref{eqn:antipode pairing} and the fact that the antipode is an anti-homomorphism, we may rewrite the expression above as
\begin{equation}
\label{eqn:equation 5}
\text{LHS of \eqref{eqn:antipode}} = \sum_{\{1,\dots,n\} = \{a_1 < \dots < a_m\} \sqcup \{b_1 < \dots < b_{n-m}\}} 
\end{equation}
$$
\sum_{\{\ell_{x,y} \geq 0\}_{\forall a_x> b_y}} \prod_{1 \leq x \leq m, 1 \leq y \leq n-m}^{a_x > b_y} \gamma_{i_{a_x}i_{b_y}}^{(\ell_{x,y})} \left \langle \prod_{x=1}^m e_{i_{a_x}, d_{a_x}+\sum_{1 \leq y \leq n-m}^{a_x>b_y} \ell_{x,y}} \prod \psi , F' \right \rangle 
$$
$$
\left \langle \prod_{y=n-m}^{1} S^{-1} \left( e_{i_{b_y},d_{b_y}-\sum_{1\leq x \leq m}^{a_x>b_y} \ell_{x,y}} \right) \prod \psi, F'' \ph  \right \rangle 
$$
where $\prod_{y=n-m}^1$ indicates that the leftmost factor in the product corresponds to $y=n-m$ and the rightmost factor corresponds to $y=1$. If we apply $S^{-1}$ to \eqref{eqn:antipode e}, we observe that for all $i \in I$ and $d \in \BZ$
$$
S^{-1}(e_i(z)) = - e_i(z)  \bar{\ph}^+_i(z) \quad \Rightarrow \quad S^{-1}(e_{i,d}) = - \sum_{k=0}^{\infty} e_{i,d-k} \bar{\ph}_{i,k}^+
$$
where we write $\left(\ph^+_i(z)\right)^{-1} =: \bar{\ph}^+_i(z) = \sum_{k=0}^{\infty} \frac {\bar{\ph}^+_{i,k}}{z^k}$. Therefore, \eqref{eqn:equation 5} becomes
\begin{equation}
\label{eqn:equation 6}
\text{LHS of \eqref{eqn:antipode}} = \sum_{\{1,\dots,n\} = \{a_1 < \dots < a_m\} \sqcup \{b_1 < \dots < b_{n-m}\}} 
\end{equation}
$$
\sum_{\{\ell_{x,y} \geq 0\}_{\forall a_x> b_y}} \prod_{1 \leq x \leq m, 1 \leq y \leq n-m}^{a_x > b_y} \gamma_{i_{a_x}i_{b_y}}^{(\ell_{x,y})} \left \langle \prod_{x=1}^m e_{i_{a_x}, d_{a_x}+\sum_{1 \leq y \leq n-m}^{a_x>b_y} \ell_{x,y}} \prod \psi , F' \right \rangle 
$$
$$
(-1)^{n-m} \sum_{k_1,\dots,k_{n-m} = 0}^{\infty} \left \langle \prod_{y=n-m}^{1} \left( e_{i_{b_y},d_{b_y}-k_y-\sum_{1\leq x \leq m}^{a_x>b_y} \ell_{x,y}} \bar{\ph}_{i,k_y}^+ \right) \prod \psi, F'' \ph \right \rangle 
$$
Using formula \eqref{eqn:rel 1 loop}, we have
\begin{equation}
\label{eqn:phi bar past e 2}
\bar{\ph}_j^+(y) e_i(x) = e_i(x) \bar{\ph}_j^+(y) \frac {\zeta_{ij} \left(\frac xy \right)}{\zeta_{ji} \left( \frac yx \right)}  \Rightarrow  \bar{\ph}_{j,k}^+ e_{i,d} = \sum_{\ell=0}^k \bar{\gamma}_{ij}^{(\ell)} e_{i,d+\ell} \bar{\ph}_{j,k-\ell}^+
\end{equation}
where the complex numbers $\bar{\gamma}_{ij}^{(\ell)}$ are defined by
\begin{equation}
\label{eqn:expansion zeta 2}
\frac {\zeta_{ij} \left(\frac xy \right)}{\zeta_{ji} \left( \frac yx \right)} = \sum_{\ell=0}^{\infty} \bar{\gamma}_{ij}^{(\ell)} \frac {x^\ell}{y^\ell}
\end{equation}
We can use \eqref{eqn:phi bar past e 2} to move $\bar{\ph}$'s to the right in the third row of \eqref{eqn:equation 6}, so we have
\begin{equation}
\label{eqn:equation 7}
\text{LHS of \eqref{eqn:antipode}} = (-1)^{n-m} \sum_{\{1,\dots,n\} = \{a_1 < \dots < a_m\} \sqcup \{b_1 < \dots < b_{n-m}\}} 
\end{equation}
$$
\sum_{\{\ell_{x,y} \geq 0\}_{\forall a_x> b_y}} \prod_{1 \leq x \leq m, 1 \leq y \leq n-m}^{a_x > b_y} \gamma_{i_{a_x}i_{b_y}}^{(\ell_{x,y})} \left \langle \prod_{x=1}^m e_{i_{a_x}, d_{a_x}+\sum_{1 \leq y \leq n-m}^{a_x>b_y} \ell_{x,y}} \prod \psi , F' \right \rangle 
$$
$$
 \sum_{k_1,\dots,k_{n-m} = 0}^{\infty} \sum_{\{\bar{\ell}_{y,y'} \geq 0\}_{1 \leq y < y' \leq n-m}} \prod_{1 \leq y < y' \leq n-m} \bar{\gamma}_{i_{b_y}i_{b_{y'}}}^{(\bar{\ell}_{y,y'})} 
$$
$$
\left \langle \prod_{y=n-m}^{1}  \bar{\ph}_{i_{b_y},k_y - \sum_{y' = 1}^{y-1} \bar{\ell}_{y',y}}^+  ,  \ph  \right \rangle  \left \langle \prod_{y=n-m}^{1} e_{i_{b_{y}},d_{b_{y}}-k_{y}-\sum_{1\leq x \leq m}^{a_x>b_{y}} \ell_{x,y}+\sum_{y'=y+1}^{n-m} \bar{\ell}_{y,y'}} \prod \psi, F'' \right \rangle  
$$
where in the last row we used \eqref{eqn:ph in pairing}. Recall from \eqref{eqn:coproduct shuffle minus} that
\begin{equation}
\label{eqn:recall coproduct}
F' \otimes F'' \ph = \frac {F(z_{i1},\dots , z_{im_i} \otimes z_{i,m_i+1}, \dots, z_{in_i})_{i \in I} \prod^{j \in I}_{1 \leq b \leq m_j} \ph^-_j(z_{jb})}{\prod^{i \in I}_{1\leq a \leq m_i} \prod^{j \in I}_{m_j < b \leq n_j} \zeta_{ij} \left( \frac {z_{ia}}{z_{jb}} \right)}
\end{equation}
with the variables expanded as $|z_{i1}|,\dots,|z_{im_i}| \ll |z_{i,m_i+1}|,\dots,|z_{in_i}|$ (we suppress the implied summation signs). Therefore, we use formulas \eqref{eqn:pairing ph} and \eqref{eqn:pairing shuffle formula} in order to evaluate the expressions on the second and third lines above, and we obtain
\begin{equation}
\label{eqn:equation 8}
\text{LHS of \eqref{eqn:antipode}} = (-1)^{n-m} \sum_{\{1,\dots,n\} = \{a_1 < \dots < a_m\} \sqcup \{b_1 < \dots < b_{n-m}\}} \sum_{\{\ell_{x,y} \geq 0\}_{\forall a_x> b_y}}
\end{equation}
$$
\mathop{\prod_{1 \leq x \leq m}^{a_x > b_y}}_{1 \leq y \leq n-m} \gamma_{i_{a_x}i_{b_y}}^{(\ell_{x,y})} \int_{1 \gg |z_{a_1}| \gg \dots \gg |z_{a_{m}}|} \frac {\prod_{x=1}^m z_{a_x}^{d_{a_x} + \sum_{1\leq y \leq n-m}^{a_x>b_y} \ell_{x,y}} F'(z_{a_1},\dots,z_{a_m})}{\prod_{1 \leq x < x' \leq m} \zeta_{i_{a_{x'}}i_{a_x}} \left(\frac {z_{a_{x'}}}{z_{a_x}} \right)} \prod_{x=1}^{m} \psi_{i_{a_x}}(z_{a_x})
$$
$$
\sum_{y_1,\dots,y_{n-m} = 0}^{\infty} \sum_{\{\bar{\ell}_{y,y'} \geq 0\}_{1 \leq y < y' \leq n-m}} \prod_{1 \leq y < y' \leq n-m} \bar{\gamma}_{i_{b_y}i_{b_{y'}}}^{(\bar{\ell}_{y,y'})} \prod_{1 \leq x \leq m, 1 \leq y \leq n-m} \frac {\zeta_{i_{a_x}i_{b_y}} \left(\frac {z_{a_x}}{z_{b_y}}\right)}{\zeta_{i_{b_y}i_{a_x}} \left(\frac {z_{b_y}}{z_{a_x}}\right)}
$$
$$
\int_{1 \ll |z_{b_1}| \ll \dots \ll |z_{b_{n-m}}|}  \frac {\prod_{y=1}^{n-m} z_{b_y}^{d_{b_y} - \sum_{y' = 1}^{y-1} \bar{\ell}_{y',y}-\sum_{1\leq x \leq m}^{a_x>b_{y}} \ell_{x,y}+\sum_{y'=y+1}^{n-m} \bar{\ell}_{y,y'}} F''(z_{b_1},\dots,z_{b_{n-m}})}{\prod_{1 \leq y < y' \leq n-m} \zeta_{i_{b_y}i_{b_{y'}}} \left(\frac {z_{b_{y}}}{z_{b_{y'}}} \right)} \prod_{y=1}^{n-m} \psi_{i_{b_y}}(z_{b_y})
$$
If we recall the definition of the coefficients $\gamma$ and $\bar{\gamma}$ from \eqref{eqn:expansion zeta 1} and \eqref{eqn:expansion zeta 2}, we obtain
\begin{equation}
\label{eqn:equation 9}
\text{LHS of \eqref{eqn:antipode}} = (-1)^{n-m} \sum_{\{1,\dots,n\} = \{a_1 < \dots < a_m\} \sqcup \{b_1 < \dots < b_{n-m}\}} 
\end{equation}
$$
\int_{|z_{a_{m}}| \ll \dots \ll |z_{a_1}| \ll 1 \ll |z_{b_1}| \ll \dots \ll |z_{b_{n-m}}|}  \prod_{1 \leq x \leq m, 1 \leq y \leq n-m}^{a_x < b_y} \frac {\zeta_{i_{a_x}i_{b_y}} \left( \frac {z_{a_x}}{z_{b_y}} \right)}{\zeta_{i_{b_y}i_{a_x}} \left( \frac {z_{b_y}}{z_{a_x}} \right)}
$$
$$
\frac {\prod_{x=1}^m z_{a_x}^{d_{a_x}}  F'(z_{a_1},\dots,z_{a_m})}{\prod_{1 \leq x < x' \leq m} \zeta_{i_{a_{x'}}i_{a_x}} \left(\frac {z_{a_{x'}}}{z_{a_x}} \right)} \frac {\prod_{y=1}^{n-m} z_{b_y}^{d_{b_y}} F''(z_{b_1},\dots,z_{b_{n-m}})}{\prod_{1 \leq y < y' \leq n-m} \zeta_{i_{b_{y'}}i_{b_y}} \left(\frac {z_{b_{y'}}}{z_{b_{y}}} \right)} \prod_{x=1}^{m} \psi_{i_{a_x}}(z_{a_x}) \prod_{y=1}^{n-m} \psi_{i_{b_y}}(z_{b_y}) 
$$
Once we recall formula \eqref{eqn:recall coproduct}, the right-hand side of the formula above is precisely the same as the right-hand side of \eqref{eqn:antipode}, which concludes the proof of Lemma \ref{lem:antipode}.

\end{proof}

\bigskip

\section{Grothendieck rings and extended $QQ$-systems}

We establish a ring isomorphism between the Grothendieck rings of $\mathcal{O}$ and $\mathcal{O}^{\sh}$, 
which is also compatible with respect to renormalization. This allows to formulate the conjectures in \cite{FH3} on generalized $QQ$-system in terms of the Borel category $\mathcal{O}$. 
As an application, we also establish an explicit solution of the $QQ$-system in the Borel category $\mathcal{O}$. This generalizes the result of \cite{FH2} from finite type to an arbitrary symmetrizable Kac-Moody Lie algebra $\fg$.

\subsection{Normalization factors}

Let us recall the normalization factors $\chi^{\br}$ that appear in \eqref{eqn:chi decomposition intro}, which in the present Section will be denoted by $\chi^\mu$ (with respect to the correspondence $\br \leftrightarrow \mu$ given by
$$
\br = (r_i \in \BZ)_{i \in I} \quad \leftrightarrow \quad \mu = \sum_{i \in I} r_i \omega_i^\vee 
$$
which will be in force throughout the present Section). The factor $\chi^{\bs^i} = \chi^{\omega_i^\vee}$ is the character of a positive prefundamental module, which corresponds to a prefundamental $\ell$-weight
\begin{equation} 
\label{eqn:ell weight1}
\bpsi_{i,a} =  \left(1,\dots,1,1- \frac az, 1, \dots, 1 \right) 
\end{equation}
with the non-trivial term situated on the $i$-th position. When $\fg$ is of finite type, it was calculated using limits of Kirillov-Reshetikhin modules in \cite{HJ}. On the other hand, in \cite{MY}, this factor was conjectured to be given by an explicit product formula over positive roots (the conjecture was proved case-by-case in all types except $E_8$, in several papers). The aforementioned formula was later generalized by \cite{W1} to the following conjecture
\begin{equation}
\label{eqn:conj factor}
\chi^\mu = \prod_{\alpha \in \Delta_+} \left(\frac 1{1 - [-\alpha]} \right)^{\text{max}(0,(\mu,\alpha))}
\end{equation}
where $\Delta_+$ is the set of positive roots of the finite type Lie algebra $\fg$, and 
$$
[-\alpha] = \left[ \Big( q^{-(\alpha,\bs^i)} \Big)_{i \in I} \right]
$$
is an $I$-tuple of constant power series. Formula \eqref{eqn:conj factor} was proved in \cite{N Char}. In fact, \emph{loc. cit.} shows that for all symmetrizable Kac-Moody Lie algebras $\fg$, we have
\begin{equation}
\label{eqn:conj factor general}
\chi^\mu = \prod_{\bn \in \nn} \left(\frac 1{1 - [-\bn]} \right)^{\text{certain exponents}}
\end{equation}
where the exponents have a shuffle algebra interpretation, and are expected to depend only on the horizontal subalgebra $\CB_{\boldsymbol{0}} \subset \UU$ (see \cite[Sections 1.4 and 3.1]{N Char} for details, and for a conjecture on the exponents in Kac-Moody types).

\subsection{Coproducts} Let $\Delta'$ be the Drinfeld-Jimbo coproduct on $\UU \cong \UUaff$ for finite type $\fg$, while for an arbitrary symmetrizable Kac-Moody Lie algebra $\fg$ we let 
\begin{equation}
\label{eqn:new}
\Delta' : \UU \rightarrow \UU \ \widehat{\otimes} \ \UU
\end{equation}
be the new new topological coproduct defined in \cite{N New} (the definition is unambiguous, since the new new coproduct was shown in \emph{loc. cit.} to match the Drinfeld-Jimbo coproduct in finite types). We will not recall the specific completion necessary in \eqref{eqn:new}, but suffice it to say that it gives rise to a well-defined action
$$
V,W \in \CO \quad \leadsto \quad V \otimes W \in \CO
.$$
We have the following formula for all $i \in I$
\begin{equation} 
\label{eqn:delta prime}
\Delta'(\ph_i(z)) \in \ph_i(z) \otimes \ph_i(z) + \UUm \otimes \UUp 
\end{equation}
which was proved for finite type $\fg$ in \cite{Da0} and for arbitrary symmetrizable Kac-Moody $\fg$ in \cite{N New}. Equation \eqref{eqn:delta prime} implies that $\ph_i(z)$ acts on $V \otimes W$ in a block triangular fashion (with respect to $V_\bpsi \otimes W_{\bpsi'}$ for various $\ell$-weights $\bpsi,\bpsi'$), with $\ph_i(z) \otimes \ph_i(z)$ on the diagonal blocks. Thus, by the argument of \cite{FR}, we conclude that
\begin{equation}
\label{eqn:q-char multiplicative}
\chi_q(V \otimes W) = \chi_q(V) \cdot \chi_q(W)
\end{equation}
for all $V,W \in \CO$.

\subsection{Grothendieck rings}

Consider the Grothendieck rings $K_0(\mathcal{O})$, $K_0(\mathcal{O}^{\sh})$. The $q$-character morphisms for category $\mathcal{O}$ and for category $\mathcal{O}^{\sh}$ are injective and have the same image (the arguments in \cite[Section 9.1]{GHL} and \cite[Theorem 4.19]{W2} hold for arbitrary symmetrizable Kac-Moody Lie algebras). 

\begin{proposition} There is a natural isomorphism of topological rings
$$
I : K_0(\mathcal{O}) \ \xrightarrow{\sim} \ K_0(\mathcal{O}^{\esh}),
$$
which commutes with the $q$-character morphisms.
\end{proposition}

Note that the multiplicative structure is defined in a different way for $\mathcal{O}$ and for  $\mathcal{O}^{\sh}$: for the first one using the new new coproduct $\Delta'$, while for the second one using the fusion procedure derived from the Drinfeld coproduct $\Delta$, see \cite{H Shifted}. It would be interesting to compare these two operations at the level of categories $\CO$.

\medskip 

Recall  the Grothendieck ring $\mathcal{E}$ of the subcategory of $\mathcal{O}$ consisting of modules with constant $\ell$-weights (see \cite[Section 2.3]{FH3} for instance). There is an analogous (and equivalent) subcategory in $\mathcal{O}^{\sh}$. Its
simple objects are the one-dimensional invertible representations $[\omega]$ parameterized by weights $\omega$. Thus, as in \cite[Section 9.7]{Kac}, we will regard elements of $\mathcal{E}$ as formal sums
\[
 c = \sum_{\omega\in \text{Supp}(c)} c(\omega)[\omega].
\]
The multiplication is given by $[\omega][\omega'] = [\omega+\omega']$ and $\mathcal{E}$
is regarded as a subring of $K_0(\mathcal{O})$ (resp. of $K_0(\mathcal{O}^{\sh})$).

\medskip 

Each $\chi^\mu$ of \eqref{eqn:conj factor general} can be seen as an element of $\mathcal{E}$, and we may consider the subring $A$ that they generate. Then, we can realize $A$ as a subring of $K(\mathcal{O})$ and of $K_0(\mathcal{O}^{\sh})$. This makes $K_0(\mathcal{O})$ and $K_0(\mathcal{O}^{\sh})$ into $A$-modules. The morphism $I$ is in fact an isomorphism of $A$-modules.

\medskip

Formula \eqref{eqn:chi decomposition intro} implies the following (recall that $\chi^\mu = \chi^{\br}$ with $\mu = \sum_{i \in I} r_i \omega_i^\vee$).

\begin{theorem}\label{isomgr} For $L(\bpsi)$ simple in $\mathcal{O}$ with corresponding shift $\mu$, we have
$$I([L(\bpsi)]) = \chi^{\mu} [L^{\esh}(\bpsi)].$$
In particular, $I$ is an isomorphism of rings preserving the bases of simple classes (up to invertible factors in $A$).
\end{theorem}

\subsection{Finite length}

Assume that $\fg$ is of finite type throughout the present subsection. Recall the following.

\begin{theorem}\label{fls}\cite{HZ} The subcategory $\mathcal{O}^{\esh}_f$ of modules of finite length in $\mathcal{O}^{\esh}$ is stable under fusion product when $\fg$ is of finite type.
\end{theorem}

In other words, the subgroup of $K_0(\mathcal{O}^{\sh})$ generated by simple classes is a subring. It will be denoted by $K_0(\mathcal{O}^{\sh}_f)$. Note that the analogous statement would not be true in $\mathcal{O}$: for example, the tensor product of a positive prefundamental module $L(\bpsi_{i,a})$ and a negative prefundamental module $L(\bpsi_{j,b}^{-1})$ (of the quantum affine Borel algebra) is not of finite length. However, by our results, we obtain the following analog of the Jordan-H\"older property for $\mathcal{O}$.

\begin{theorem} The sub $A$-module of $K_0(\mathcal{O})$ generated by simple classes is a subring of $K_0(\mathcal{O})$.
\end{theorem}

Consider the subgroup $\overline{K}_0(\mathcal{O})$ of $K_0(\mathcal{O})$ generated by the $\chi^{-\mu}[L(\Psi)]$, as $\Psi$ runs over $\ell$-weights and $\mu = \sum_{i \in I} r_i\omega_i^\vee$ for $\br = \ord \bpsi$. We also obtain the following.

\begin{theorem} $\overline{K}_0(\mathcal{O})$ is a subring of $K_0(\mathcal{O})$ isomorphic to 
$K_0(\mathcal{O}^{\esh}_f)$.
\end{theorem}

Let us call $\overline{K}_0(\mathcal{O})$ the renormalized Grothendieck ring of the category $\mathcal{O}$. It allows to study the ring structures in the shifted and Borel cases on equal footing. For example, in simply-laced types, we can also reformulate the monoidal categorification conjecture of \cite{GHL} for $\mathcal{O}^{\sh}$ in terms of $\mathcal{O}$. It is proved in \emph{loc. cit.} that there is an embedding of a cluster algebra of infinite rank $\mathcal{A}$ into $K_0(\mathcal{O}^{\sh})$
$$i : \mathcal{A}\rightarrow K_0(\mathcal{O}^{\sh}).$$
Moreover, the closure of the image of $i$ is $K_0(\mathcal{O}^{\sh})$ (more precisely, an integral subcategory $\mathcal{O}^{\sh}_{\mathbb{Z}}$ 
should be used instead of $\mathcal{O}^{\sh}$, but we will abuse notation and use the same symbol). By the discussion above, we have also an embedding
$$I^{-1}\circ i : \mathcal{A}\rightarrow K_0(\mathcal{O}).$$
The main Conjecture of \cite{GHL} states that the cluster monomials in $\mathcal{A}$ should correspond to simple classes in $K_0(\mathcal{O}^{\sh})$.
This can now be reformulated as follows.

\begin{conjecture}  The image by $I^{-1}\circ i$ of the cluster monomials in $\mathcal{A}$ are simple classes up to a factor $\chi^\mu$ and belong to $\overline{K}_0(\mathcal{O})$.
\end{conjecture}

\subsection{$QQ$-systems for finite types}

Still assuming that $\fg$ is of finite type, recall the $QQ$-systems \cite{FH2} and their extended versions \cite{FH3} in the Grothendieck ring $K_0(\mathcal{O})\simeq K_0(\mathcal{O}^{\sh})$. This is a system of algebraic relations
\begin{equation}
\label{eqn:qq system}
\mathcal{Q}_{ws_i(\omega_i^\vee),aq_i} \mathcal{Q}_{w(\omega_i^\vee),aq_i^{-1}} - 
[-w(\alpha_i)]  \mathcal{Q}_{ws_i(\omega_i^\vee),aq_i^{-1}}\mathcal{Q}_{w(\omega_i^\vee),aq_i}
\end{equation}
$$
= [-w(\alpha_i)]_+ \prod_{j \neq i} \prod_{s \in \{c_{ij} + 1, c_{ij}+3,\dots,-c_{ij}-3,-c_{ij}-1\}} \mathcal{Q}_{w(\omega_j^\vee),aq_i^s}
$$
where variables $\mathcal{Q}_{w(\omega_i^\vee),a}$ depend on $i\in I$, $a\in\mathbb{C}^*$ and a Weyl group element $w$ (here we consider the Weyl group of the underlying Lie algebra; it is isomorphic to the Weyl group its Langlands dual Lie algebra). We also use the notation  $[-w(\alpha_i)]_+ = 1$ if $w(\alpha_i)\in\Delta_+$ and $[-w(\alpha_i)]_+ = - [-w(\alpha_i)]$ if $w(\alpha_i)\in \Delta_-$, where $\Delta_\pm$ refers to a choice of positive/negative roots of $\fg$.

\medskip

A solution of the $QQ$-system was constructed in \cite{FH3} in $K_0(\mathcal{O})\simeq K_0(\mathcal{O}^{\sh})$, as follows: let us recall the prefundamental $\ell$-weight
\begin{equation} 
\label{eqn:ell weight}
\bpsi_{\omega_i^\vee,a} = \bpsi_{i,a} =  \left(1,\dots,1,1- \frac az, 1, \dots, 1 \right) 
\end{equation}
with the non-trivial term situated on the $i$-th position. Although we will not recall the formula, for any Weyl group element $w \in W$ one can define an $\ell$-weight
\begin{equation} 
\label{eqn:w ell weight}
\bpsi_{w(\omega_i^\vee),a}
\end{equation}
as in \cite{FH3} (in a nutshell, $\bpsi_{w(\omega_i^\vee),a}$ is constructed in \cite{FH3} by using an extension of the Chari braid group action \cite{C}; the shift associated to this $\ell$-weight is  $w(\omega_i^\vee)$).

\begin{conjecture}
\label{conj:qq shifted}

\cite[Conjecture 6.11]{FH3} For any $i\in I$, $a\in\mathbb{C}^*$, $w\in W$ we have
$$
\mathcal{Q}_{w(\omega_i^\vee),a} = [L^{\esh}(\bpsi_{w(\omega_i^\vee),a})]
$$
in $K_0(\mathcal{O}^{\esh})$.
\end{conjecture}

The results of the previous section allow us to infer that Conjecture \ref{conj:qq shifted} is equivalent to the following.

\begin{conjecture} 
\label{conj:qq borel}

For any $i\in I$, $a\in\mathbb{C}^*$, $w\in W$ we have the relation 
$$
\mathcal{Q}_{w(\omega_i^\vee),a} = (\chi^{w(\omega_i^\vee)})^{-1}[L(\bpsi_{w(\omega_i^\vee),a})]
$$
in $\overline{K}_0(\mathcal{O})$. 
\end{conjecture}

Conjecture \ref{conj:qq borel} is a more precise formulation of the conjectural solutions of the $QQ$-system in terms of the representation theory of quantum affine Borel algebra. Also, the second part of \cite[Conjecture 6.8]{FH3} for the character of $L^{\sh}(\bpsi_{w(\omega_i^\vee),a})$ can be reformulated as the following. 
Recall the series $\chi_{w(\omega_i^\vee)}$ introduced in \cite{FH3}. 

\begin{conjecture}\label{charexp} For $w\in W$, $i\in I$ and $a\in\mathbb{C}^*$, we have
$$
\chi(L(\bpsi_{w(\omega_i^\vee),a})) = \chi^{w(\omega_i^\vee)} \chi_{w(\omega_i^\vee)}.
$$

\end{conjecture}

For $w = s_i$ a simple reflection, as we know the $q$-character of $L^{\sh}(\bpsi_{s_i(\omega_i^\vee),a})$ by \cite[Example 5.2]{H Shifted}, we obtain the $q$-character and the character of  $L(\bpsi_{s_i(\omega_i^\vee),a})$. In particular, we have
$$
\chi(L(\bpsi_{s_i(\omega_i^\vee),a})) = \chi^{s_i(\omega_i^\vee)}\frac{1}{1 - [-\alpha_i]}.
$$
Thus, Conjecture \ref{charexp} is established for $w=s_i$. Moreover, this allows to make all the constants precise for the $QQ$-system established in \cite{FH2}. Consider the simple classes
$$
Q_{i,a} = [L(\bpsi_{i,a})] \quad \text{and} \quad   
\tilde{Q}_{i,a} = [L(\tilde{\bpsi}_{i,aq_i^{-2}})]
$$
with the notation as in \cite{FH2}. Indeed, it is established in \emph{loc. cit.} that we have in $K_0(\mathcal{O})$ the relation
\begin{equation}\label{QQ}\tilde{Q}_{i,aq_i} Q_{i,aq_i^{-1}} - 
[-\alpha_i]  \tilde{Q}_{i,aq_i^{-1}}Q_{i,aq_i}\end{equation}
$$
= \chi \prod_{j \neq i} \prod_{s \in \{c_{ij} + 1, c_{ij}+3,\dots,-c_{ij}-3,-c_{ij}-1\}} Q_{w(\omega_j^\vee),aq_i^s} 
$$
for a constant $\chi$. 

\begin{theorem} The constant in the $QQ$-system (\ref{QQ}) is equal to
\begin{equation}\label{solc}\chi = \chi^{\omega_i^\vee}\chi^{s_i(\omega_i^\vee)}(\chi^{s_i(\omega_i^\vee) + \omega_i^\vee})^{-1}.\end{equation}
\end{theorem}

\subsection{$QQ$-systems for general types} Let us now assume that $\fg$ is an arbitrary symmetrizable Kac-Moody Lie algebra. We conjecture that all the formulas in the preceding Subsection remain valid in this new generality. Specifically, we write (see \eqref{eqn:a weight})
\begin{equation}
\label{eqn:a is ratio}
A_{i,a}^{-1} = \frac {\left[{\widetilde{\bpsi}_{i,aq^{-2}_i}\bpsi_{i,a}^{-1}}\right]}{\left[{\widetilde{\bpsi}_{i,a}\bpsi_{i,aq_i^2}^{-1}}\right]}
\end{equation}
where the $\ell$-weight 
$$
\tilde{\bpsi}_{i,a} \quad \text{has }j\text{-component} \quad \begin{cases} \prod_{s \in \{{c_{ij}+2,c_{ij}+4,\dots,-c_{ij}-2,-c_{ij}}\}} \left(1-\frac {aq_i^s}z \right) &\text{if } j\neq i \\ {\left(1-\frac{a}{z} \right)^{-1}} &\text{if }j=i \end{cases}
$$
and one can run the machinery of \cite{FH, FH2, FH3} without modifications. In particular, we obtain a generalization to general symmetrizable Kac-Moody Lie algebras of the main result of \cite{FH2}.

\begin{theorem}
\label{thm:qq} 

For a general symmetrizable Kac-Moody Lie algebra, we have a solution of the $QQ$-relation 
(\ref{QQ}) in $K_0(\mathcal{O})$ given by the simple classes $Q_{i,a} = [L(\bpsi_{i,a})]$ and 
$\tilde{Q}_{i,a} = [L(\tilde{\bpsi}_{i,aq_i^{-2}})]$. The constant $\chi$ is given by \eqref{solc}.

\end{theorem}

\begin{proof} We first prove the $QQ$-relation {\it without constant} in $K_0(\mathcal{O}^{\sh})$, 
with $Q_{i,a}$, $\tilde{Q}_{i,a}$ corresponding to simple representations. The proof is the same as in 
\cite[Section 5.4]{H Shifted}, provided the crucial relation (\ref{eqn:a is ratio}). Then we conclude 
using Theorem \ref{isomgr}.
\end{proof}


\begin{thebibliography}{XXX}

\bibitem{B}
Beck J.,
{\em Braid group action and quantum affine algebras},
Comm.\ Math.\ Phys.\ 165 (1994), no.~3, 555-568.

\bibitem{BS}
Burban I., Schiffmann O.,
{\em On the Hall algebra of an elliptic curve, I}
Duke Math. J. 161(7): 1171-1231.

\bibitem{C}
Chari V., 
{\em Braid group actions and tensor products}, 
Int. Math. Res. Not. 2002 (2002) 357-382.

\bibitem{CP94} 
Chari V. and Pressley A., 
{\em Quantum affine algebras and their representation}, 
CMS Conf. Proc. 16 (1994), 59--78.

\bibitem{CP}
Chari V., Pressley A.,
{\em A guide to quantum groups},
Cambridge University Press (1995).

\bibitem{Da0} Damiani I., 
{\em La R-matrice pour les algèbres quantiques de type affine non tordu}, 
Ann. Sci. Ecole Norm. Sup. 31 (1998), no. 4, 493-523.

\bibitem{Da} Damiani I., 
{\em From the Drinfeld realization to the Drinfeld-Jimbo presentation of affine quantum algebras : Injectivity}, Publ. Res. Inst. Math. Sci. 51 (2015), 131--171.

\bibitem{Da2} Damiani I., 
On the Drinfeld coproduct. Pure Appl. Math. Q. 20 (2024), no. 1, 171--232.

\bibitem{Dr}
Drinfeld V.,
{\em A new realization of Yangians and of quantum affine algebras}, 
(Russian) Dokl. Akad. Nauk SSSR 296 (1987), no. 1, 13–17; translation in Soviet Math. Dokl. 36 (1988), no. 2, 212–216.

\bibitem{E}
Enriquez B.,
{\em On correlation functions of Drinfeld currents and shuffle algebras}, 
Transform.\ Groups 5 (2000), no.~2, 111-120.

\bibitem{FF} 
Feigin B., Frenkel E., 
{\em Quantization of soliton systems and Langlands duality}, in  Adv. Stud. Pure Math. 61, 185-274
  Math. Soc. Japan, Tokyo, 2011.

\bibitem{FJMM} 
Feigin B., Jimbo M., Miwa T. and Mukhin E., 
{\em Finite type modules and Bethe Ansatz equations},
Ann. Henri Poincar\'e 18 (2017) 2543--2579.

\bibitem{FO}
Feigin B., Odesskii A.,
{\em Quantized moduli spaces of the bundles on the elliptic curve and their applications}, 
NATO Sci. Ser. II Math. Phys. Chem., 35, 123-137, Kluwer Acad. Publ., Dordrecht, 2001.	

\bibitem{FT} 
Finkelberg M., Tsymbaliuk A., 
{\em Multiplicative slices, relativistic Toda and shifted quantum affine algebras}, in Progr. Math. 330 (2019), 133--304.

\bibitem{FH}
Frenkel E., Hernandez D.,
{\em  Baxter’s relations and spectra of quantum integrable models}, 
Duke Math. J. 164 (2015), no. 12, 2407-2460.

\bibitem{FH2}
Frenkel E., Hernandez D.,
{\em  Spectra of quantum KdV Hamiltonians, Langlands duality, and affine opers}, 
Comm. Math. Phys. 362 (2018), no. 2, 361--414.

\bibitem{FH3}
Frenkel E., Hernandez D.,
{\em  Extended Baxter relations and QQ-systems for quantum affine algebras}, 
Comm. Math. Phys. 405 (2024), 190.

\bibitem{FM}
Frenkel E., Mukhin E.,
{\em Combinatorics of $q$-characters of finite-dimensional representations of quantum affine algebras}, 
Commun. Math. Phys. 216, 23-57 (2001).

\bibitem{FR}
Frenkel E., Reshetikhin N.,
{\em The $q$-characters of representations of quantum affine algebras and deformations of $W$-Algebras}, 
in Recent Developments in Quantum Affine Algebras and related topics, Contemp. Math. 248 (1999), 163-20.

\bibitem{GHL}
Geiss C., Hernandez D., Leclerc B.,
{\em Representations of shifted quantum affine algebras and cluster algebras I. The simply-laced case}, 
 Proc. Lond. Math. Soc. (3) 129 (2024), no. 3, Paper No. e12630.

\bibitem{HTG} Hernandez D., {\em Representations of quantum affinizations and fusion product}, 
Transform. Groups 10 (2005), no. 2, 163--200.

\bibitem{HLMS} 
Hernandez D., 
{\em Drinfeld coproduct, quantum fusion tensor category and applications}, 
Proc. Lond. Math. Soc. (3) 95 (2007), no. 3, 567--608.

\bibitem{H Shifted}
Hernandez D.,
{\em Representations of shifted quantum affine algebras},
Int. Math. Res. Not., Vol. 2023, No. 13, pp. 11035-11126.

\bibitem{H25} Hernandez D., 
{\em Representations and characters of quantum affine algebras at the crossroads between cluster categorification and quantum integrable models}, 
to appear in Proccedings of the ICM 2026 (arXiv:2510.06437).

\bibitem{HJ}
Hernandez D., Jimbo M.,
{\em Asymptotic representations and Drinfeld rational fractions}, 
Comp. Math. 2012; 148(5):1593-1623.

\bibitem{HL Borel}
Hernandez D., Leclerc B.,
{\em Cluster algebras and category $\CO$ for representations of Borel subalgebras of quantum affine algebras},
Algebra Number Theory 10(9): 2015-2052 (2016).

\bibitem{HZ}
Hernandez D., Zhang H.,
{\em Jordan-Hölder property for shifted quantum affine algebras}, Preprint arXiv:2501.16859.

\bibitem{Kac}
Kac V., 
{\em Infinite dimensional Lie algebras}, Third Edition, Cambridge University Press, Cambridge, New York, 1990.

\bibitem{mrv}
Masoero D., Raimondo A., Valeri D., 
{\em Bethe Ansatz and the Spectral Theory of affine Lie algebra valued connections}, 
Commun. Math. Phys. 344 (2016) 719--750.

\bibitem{MY}
Mukhin E., Young C., 
{\em Affinization of category $\CO$ for quantum groups}, 
Trans. Amer. Math. Soc. 366 (2014), no. 9, 4815-4847.

\bibitem{N Loop}
Negu\cb{t} A.,
{\em Quantum loop groups for symmetric Cartan matrices}, 
J. Reine Angew. Math. (2026).

\bibitem{N Arbitrary}
Negu\cb{t} A.,
{\em Quantum loop groups for arbitrary quivers}, 
From representation theory to mathematical physics and back, 287–324. Contemp. Math., 817, American Mathematical Society, Providence RI, 2025.

\bibitem{N Cat}
Negu\cb{t} A.,
{\em Category $\CO$ for quantum loop algebras},
ar$\chi$iv:2501.00724.

\bibitem{N Char}
Negu\cb{t} A.,
{\em Characters of quantum loop algebras},
ar$\chi$iv:2503.17518.

\bibitem{N New}
Negu\cb{t} A.,
{\em A new new coproduct on quantum loop algebras},
ar$\chi$iv:2602.01130.

\bibitem{NT}
Negu\cb{t} A., Tsymbaliuk A.,
{\em Quantum loop groups and shuffle algebras via Lyndon words}, 
Adv.\ Math.\ 439 (2024), 109482, 69 pp.

\bibitem{W2}
Wang K.
{\em QQ-systems for twisted quantum affine algebras}, 
Commun. Math. Phys. 400 (2023), 1137--1179.

\bibitem{W1}
Wang K.,
{\em Weyl group twists and representations of quantum affine Borel algebras},
Algebr. Represent. Theor. (2025).



\end{thebibliography}
\end{document}